\documentclass[12pt,reqno]{amsart} 

\usepackage[hmarginratio={1:1},scale=0.8]{geometry}

\usepackage{graphicx}%
\usepackage{multirow}%
\usepackage{amsmath,amssymb,amsfonts}%
\usepackage{amsthm}%
\usepackage{mathrsfs}%
\usepackage[title]{appendix}%
\usepackage{xcolor}%
\usepackage{textcomp}%
\usepackage{manyfoot}%
\usepackage{booktabs}%
\usepackage{algorithm}%
\usepackage{algorithmicx}%
\usepackage{algpseudocode}%
\usepackage{listings}%

\usepackage{mathtools}
\usepackage{enumerate}
\mathtoolsset{showonlyrefs}
\usepackage{comment}
\numberwithin{equation}{section}

\newtheorem{lemma}{Lemma}[section] 
\newtheorem{theorem}[lemma]{Theorem}

\newtheorem{definition}[lemma]{Definition} 
\newtheorem{corollary}[lemma]{Corollary} 
 
\newtheorem{proposition}[lemma]{Proposition} 
\newtheorem{remark}[lemma]{Remark}

\makeatletter
\newtheorem*{rep@theorem}{\rep@title}
\newcommand{\newreptheorem}[2]{%
\newenvironment{rep#1}[1]{%
 \def\rep@title{#2 \ref{##1}}%
 \begin{rep@theorem}}%
 {\end{rep@theorem}}}
\makeatother
\newreptheorem{theorem}{Theorem}
\newreptheorem{corollary}{Corollary}

\newcommand{\R}{\mathbb{R}} 
\def \Rn{{\R^n}}

\newcommand{\N}{\mathbb{N}}

\def \s{\mathbb{S}^{n-1}}
\def \S{\mathbb{S}^{n\hid-1}}

\newcommand \B[1][n]{B_2^{#1}}
\newcommand{\vol}[1][n]{\operatorname{vol}_{#1}}
\newcommand{\Vol}{\operatorname{vol}_{nm}}
\def \s{\mathbb{S}^{n-1}}
\def \S{\mathbb{S}^{nm-1}}

\newcommand {\conbod}[1][nm] {\mathcal{K}^{#1}}
\newcommand {\conbodo}[1][nm] {\mathcal{K}^{#1}_o}
\newcommand {\conbodio}[1][nm] {\mathcal{K}^{#1}_{(o)}}

\DeclareMathOperator{\SL}{SL}
\DeclareMathOperator{\GL}{GL}

\renewcommand{\P}[1][Q]{\Pi_{#1,p}}

\newcommand{\PP}[1][Q]{\Pi_{#1,p}^{\circ}}

\newcommand{\G}[1][Q]{\Gamma_{#1,p}}

\newcommand {\LYZ}[1]{\P #1}
\newcommand {\LYZP}[1]{\PP #1}

\title[On the $m$th-order Affine {P}\'olya-{S}zeg\"o Principle]{On the $m$th-order Affine {P}\'olya-{S}zeg\"o Principle}

\author[D. Langharst]{Dylan Langharst}
\address{Institut de Math\'ematiques de Jussieu, Sorbonne Universit\'e, Paris, 75252, France}
\email{dylan.langharst@imj-prg.fr}

\author[M. Roysdon]{Michael Roysdon}
\address{Department of Mathematics, Applied Mathematics, and Statistics, Case Western Reserve University, Cleveland, OH 44106, USA}
\email{mar327@case.edu}

\author[Y. Zhao]{Yiming Zhao}
\address{Department of Mathematics, Syracuse University, Syracuse, NY 13244, USA}
\email{yzhao197@syr.edu}

\thanks{D.L. is supported by the Fondation Sciences Math\'ematiques de Paris Postdoctoral program}
\thanks{M.R. is supported by an AMS-Simons Travel Grant}
\thanks {Y.Z. is supported, in part, by U.S. National Science Foundation, Grant DMS-2337630.}

\begin{document}
\begin{abstract}
An affine {P}\'olya-{S}zeg\"o principle for a family of affine energies, with equality condition characterization, is demonstrated. In particular, this recovers, as special cases, the $L^p$ affine {P}\'olya-{S}zeg\"o principles due to Cianchi, Lutwak, Yang and Zhang, and subsequently Haberl, Schuster and Xiao. Various applications of this new {P}\'olya-{S}zeg\"o principle are shown.
\end{abstract}

\maketitle

\section{Introduction}

The sharp $L^p$ Sobolev inequality states that for $p\in (1,n)$ and $f\in W^{1,p}(\Rn)$, we have
\begin{equation}
	\label{eq:lp_sobolev}
a_{n,p}\|f\|_{L^{\frac{np}{n-p}}(\Rn)}\leq \|\nabla f\|_{L^p(\Rn)},
\end{equation}
where $\|\cdot\|_{L^q(\Rn)}$ stands for the usual $L^q$ norm of a function. The constant $a_{n,p}$ is the best possible and can be explicitly computed, see \eqref{eq:sobolev_cons} below.  The Sobolev inequality has a wide range of applications across all areas of mathematics, particularly in the study of PDEs. In this sharp form, it can be found in Federer and Fleming \cite{FF60}, Fleming and 
Rishil \cite{MR114892}, Maz'ya \cite{VGM60} for $p=1$, and Aubin \cite{Aubin1} and Talenti \cite{Talenti1} for $p\in (1,n)$. The Sobolev inequality is arguably the most fundamental inequality connecting analysis and geometry. Indeed, the geometric core behind \eqref{eq:lp_sobolev}, for all $p\in (1,n)$, is the classical isoperimetric inequality:
\begin{equation}
	n\omega_n^{\frac{1}{n}}\text{volume} (E)^{\frac{n-1}{n}}\leq \text{surface area}(E),
\end{equation}
where $\omega_n$ is the volume of the $n$-dimensional unit ball. In fact, the two inequalities are equivalent; see, \emph{e.g.,} Gardner \cite{G20}. Moreover, the standard approaches to both inequalities are perfectly parallel. A critical ingredient for the proof of the Sobolev inequality is the \emph{{P}\'olya-{S}zeg\"o principle} stating that the $L^p$ Dirichlet energy $\|\nabla f\|_{L^p(\Rn)}$ is non-increasing if $f$ is replaced by its \emph{spherically symmetric arrangement} whereas the classical isoperimetric inequality can be shown by establishing the monotonicity of the surface area with respect to \emph{Steiner symmetrization}.  

By now, there is a vast library of literature on Sobolev inequalities and {P}\'olya-{S}zeg\"o principles. We refer the readers to the survey \cite{MR3503198} by Talenti, a few recent contributions \cite{BGGK24,MR2393140,CENV04,HJM16,LYZ06} and the references therein.

By convex body, we mean a compact convex set in $\R^n$ with non-empty interior. The \emph{projection body} $\Pi K$ of a convex body $K$ is an important and natural (see, Ludwig \cite{ML02}) object in affine geometry. As the name suggests, the projection body encodes information about the area of the image of the orthogonal projection of $K$ onto $(n-1)$-dimensional subspaces. Crucially, the fact that the projection body operator is $\SL(n)$-contravariant, or $\Pi\phi K= \phi^{-t}\Pi K$ for all $\phi\in \SL(n)$, makes its volume and the volume of its polar body fundamental \emph{affine} invariants. Recalling that, when a convex body $K$ contains the origin in its interior, its polar body $K^\circ$ is well-defined, the \emph{polar projection body} of a convex body $K$ is precisely $\Pi^\circ K :=(\Pi K)^\circ$. Then, the celebrated \emph{Petty projection inequality} states that for every convex body $K$ in $\Rn$, we have, denoting by $\vol$ the Lebesgue measure on $\Rn$,
\begin{equation}\label{eq 8.25.1}
	\vol(\Pi^\circ K) \vol(K)^{n-1}\leq \vol(\Pi^\circ \B) \vol(\B)^{n-1},
\end{equation}
where $\B$ is the centered unit ball in $\Rn$ and equality holds if and only if $K$ is an ellipsoid. By Cauchy's area formula and a simple application of H\"older's inequality, it can be seen that the Petty projection inequality \emph{trivially} implies the classical isoperimetric inequality. Quoting Gardner and Zhang \cite{GZ98}:``the Petty projection inequality is a dramatic improvement upon the classical isoperimetric inequality.'' Part of the dramatic improvements of Petty's inequality over its classical counterpart is that \eqref{eq 8.25.1} is \emph{affine invariant}, or, more precisely, changing $K$ into $\phi K$ for any $\phi\in \SL(n)$ will not change the value of the volume product $\vol(\Pi^\circ K) \vol(K)^{n-1}$. It is worth pointing out that the inequality involving the volume ratio $\vol(\Pi K) \vol(K)^{-n+1}$, $n \geq 3$, known as \emph{the Petty conjecture}, is still open \cite{CMP71}. 

In a landmark work \cite{GZ99}, Zhang, using Minkowski's existence theorem, formulated and derived from the Petty projection inequality, an \emph{affine} Sobolev inequality that is stronger than \eqref{eq:lp_sobolev}: 
\begin{equation}\label{eq 8.25.4}
	a_{n,1} \|f\|_{L^{\frac{n}{n-1}}(\Rn)}\leq \frac{(n\omega_n)^\frac{n+1}{n}}{2\omega_{n-1}} \left(\int_{\s}\left(\int_{\Rn} |(\nabla f(x))^t u| dx\right)^{-n}du\right)^{-\frac{1}{n}}: = \mathcal{E}_1(f),
\end{equation} 
where equality holds when $f$ is the characteristic function of an ellipsoid. Zhang \cite{GZ99} originally stated \eqref{eq 8.25.4} for $C^1$ functions with compact support. However, the inequality holds in general for functions of bounded variation as shown by Wang \cite{TW12}. We emphasize again that not only is \eqref{eq 8.25.4} sharp, but the energy functional $\mathcal{E}_1(f)$ is affine invariant in the sense that $\mathcal{E}_1(f)=\mathcal{E}_1(f\circ \phi)$ for $\phi\in \SL(n)$. Moreover $\mathcal{E}_1(f)\leq \|\nabla f\|_{L^1(\Rn)}$ and therefore, we may view $\mathcal{E}_1(f)$ as the \emph{affine} analog of the classical Dirichlet energy. 

In fact, Lutwak, Yang, and Zhang \cite{LYZ02} showed that \eqref{eq 8.25.4} is only one of a family of sharp $L^p$ affine Sobolev inequality: for $p\in [1,n)$,
\begin{equation}\label{eq 8.25.5}
	a_{n,p} \|f\|_{L^{\frac{np}{n-p}}(\Rn)}\leq \mathcal{E}_p(f),
\end{equation}
each of which is stronger than its classical counterpart \eqref{eq:lp_sobolev}. Here, the precise formulation of $\mathcal{E}_p(f)$ can be found in \cite{LYZ02}, or, by substituting $m=1$ and $Q=[-1/2,1/2]$ in \eqref{eq:m_lp_affine_sobolev}. Curiously, while the geometric core behind the $L^p$ Sobolev inequality for any $p\in [1,n)$ is the same (the isoperimetric inequality), the geometric core behind \eqref{eq 8.25.5} is an $L^p$ version of the Petty projection inequality (different for each $p$) established in Lutwak, Yang, and Zhang \cite{LYZ00}. An asymmetric (and even stronger) version of \eqref{eq 8.25.5} is due to Haberl and Schuster \cite{HS2009, HS09}. The proofs invariably use various elements from the booming $L^p$ Brunn-Minkowski theory initiated by Lutwak \cite{LE93,LE96} in the 1990s, including the $L^p$ Minkowski inequality and the $L^p$ Minkowski problem, see, \emph{e.g.}, \cite{CW06,HLYZ05, LYZ04}.

As in the Euclidean case, an affine {P}\'olya-{S}zeg\"o principle was established by Lutwak, Yang, and Zhang \cite{LYZ02} and Cianchi, Lutwak, Yang, and Zhang \cite{CLYZ09} for the $L^p$ affine energy: for $p\geq 1$,
\begin{equation}\label{eq 8.25.6}
	\mathcal{E}_p(f)\geq \mathcal{E}_p(f^\star),
\end{equation}
where $f^\star$ is the spherically symmetric rearrangement of $f$. See Section \ref{sec:back} for its precise definition. An asymmetric (and even stronger) version of \eqref{eq 8.25.5} is due to Haberl, Schuster, and Xiao \cite{HSX12}. An equality condition characterization for $p>1$ was found in Nguyen \cite{NVH16}.

Recently, the first two named authors and their collaborators \cite{HLPRY23,HLPRY23_2} established a higher-order version of the sharp $L^p$ affine Sobolev inequality \eqref{eq 8.25.6}. In particular, their work encompasses \cite{HS2009,LYZ02,GZ99} as special cases. While not explicitly stated, it will become clear in the current work that each inequality in this family is stronger than its classical counterpart.

One of the goals of the current work is to show the validity of the accompanying higher-order affine {P}\'olya-{S}zeg\"o principle along with its equality condition characterization.

A convex body $K$ is uniquely determined by its support function $h_K:\s\rightarrow \R$ given by $h_K(u) = \max_{x\in K} x^t u$. The \emph{covariogram function} of $K$ is given by 
\begin{equation}
	g_K(x) = \vol(K\cap (K+x)).
\end{equation}
The projection body is closely connected to the covariogram of a convex body, see Matheron \cite{MA}: for $\theta\in \s$,
\begin{equation}\label{eq 8.25.7}
	\left.\frac{d}{dr}\right|_{r=0^+} g_K(r\theta) = - \vol[n-1] (P_{\theta^{\perp}}K) = -h_{\Pi K}(u),
\end{equation}
where $P_{\theta^{\perp}}K$ is the image of the orthogonal projection of $K$ onto the $(n-1)$-dimensional subspace $\theta^\perp$ consisting of vectors perpendicular to $\theta$. 

Let $m\in \mathbb{N}$. For $x\in \R^{nm}$, we will write $x= (x_1,\dots, x_m)$ where $x_i\in \Rn$. From time to time, we shall identify $\mathbb{R}^{nm}$ with the space of $n$ by $m$ matrices. Consequently, as an example,  the notation $u^tx$ for $u\in \Rn$ and $x\in \mathbb{R}^{nm}$ is simply the usual matrix multiplication between a $1$ by $n$ row vector $u^t$ and an $n$ by $m$ matrix $x$ and the product yields a $1$ by $m$ row vector.

The motivations of \cite{HLPRY23, HLPRY23_2} stem from the  investigation of Schneider \cite{Sch70} of the the following natural generalization of the covariogram function: for each $m\in \mathbb{N}$, the \emph{$m$th-covariogram function} of a convex body $K$ in $\R^n$, is given by
\begin{equation}
	g_{K,m}(x) = \vol\left(K\cap \bigcap_{i=1}^m(K+x_i)\right),
\end{equation} 
for $x=(x_1,\dots, x_m)\in \mathbb{R}^{nm}$. This \emph{seemingly} innocent generalization turns out to lead to unexpected things. We will now describe one conjecture following this. Since the support of the classical covariogram $g_K$ gives an alternative definition for the difference body $DK= K+(-K)$, the support of $g_{K,m}$, denoted by $D^m(K)$, naturally generalizes the classical difference body. As was noted by Schneider, the set $D^m(K)$ is a convex body in $\R^{nm}$, and the volume ratio
\begin{equation}\label{eq 8.25.9}
	\vol[nm](D^m K)\vol[n](K)^{-m}
\end{equation}
is $\SL(n)$-invariant. It was shown in \cite{Sch70} that \eqref{eq 8.25.9} is maximized by simplices, which in the case of $m=1$ recovers the celebrated difference body inequality of Rogers and Shephard \cite{RS57}, a certain reverse form of the celebrated Brunn-Minkowski inequality \cite{G20}.  From this point of view, 
finding the \emph{sharp} lower bound of \eqref{eq 8.25.9} may be viewed as a higher-order analog of the Brunn-Minkowski inequality. 
The conjecture is settled in dimension 2: Schneider showed that it is minimized by origin-symmetric $K$ when $n =2$ for all $m \in \N$. The case for arbitrary $n$ when $m=1$ is the Brunn-Minkowski inequality for the difference body. The minimizer of \eqref{eq 8.25.9} in the case $n \geq 3, m \geq 2$ remains open and is conjectured to be attained by ellipsoids. With this point of view in mind, it seems apt to refer to Schneider's conjecture as the {\it mth-order Brunn-Minkowski conjecture}.  As observed in \cite{Sch20}, this conjecture is intimately connected to Petty's conjecture mentioned above. With the tools developed in \cite{HLPRY23}, Haddad in \cite{JH23} confirmed a version of the  $m$th-order Brunn-Minkowski conjecture by replacing the volume with mean width. 
 
It was shown in \cite{HLPRY23} that $g_{K,m}$ is differentiable in each radial direction at the origin. Inspired by Matheron's formula \eqref{eq 8.25.7}, one naturally obtains, for each $m\in \mathbb{N}$, an $m$th-order projection body; this corresponds to the case when $Q$ is the orthogonal simplex in the following definition, introduced in \cite{HLPRY23_2}.
\begin{definition} \label{d:generalprojectionbody} Let $p \geq 1$, $m \in \N$, and fix a convex body $Q$ containing the origin in $\R^m$. Given a convex body $K$ containing the origin in its interior in $\R^n$, we define the $(L^p, Q)$-projection body of $K$, $\P K $, to be the convex body whose support function is given by 
\[
h_{\P K}(u)^p = \int_{\s} h_Q(v^tu)^p d\sigma_{K,p}(v)
\]
for $u\in\S$.
\end{definition}
Here $\sigma_{K,p}$ is the $L^p$ surface area measure of $K$, see Section \ref{sec:back} for its definition. It is simple to see that $\P$ is $\SL(n)$ ``contravariant''; that is, $\P \phi K = \overline{\phi^{-t}} \P K$ for each $\phi\in \SL(n)$, where $\overline{\phi^{-t}}(x_1,\dots, x_m)= (\phi^{-t} x_1, \dots, \phi^{-t} x_m)$. Moreover, when $m=1, Q=[-1/2,1/2]$, it recovers the classical $L_p$ projection body $\Pi_p K$ and when $m=1, Q=[0,1]$, it recovers the asymmetric $L^p$ projection body $\Pi_p^+K$ first considered in Lutwak \cite{LE96}. It was shown in \cite{HLPRY23_2}  that $\P K$ contains the origin as an interior point. Consequently, its polar body $(\P K)^\circ$ will simply be denoted by $\PP K$. 

In \cite{HLPRY23_2}, a brand new family of sharp affine isoperimetric inequality (containing the Petty projection inequality as a special case) was shown: for $p\geq 1$, $m\in \mathbb{N}$, and a convex body $K$ in $\Rn$ containing the origin in its interior,
\begin{equation}\label{eq 8.25.8}
	\Vol(\PP K) \vol(K)^{\frac{nm}{p}-m} \leq \Vol(\PP \B) \vol(\B)^{\frac{nm}{p}-m}.
\end{equation}
For $p=1$, equality holds if and only if $K$ is an ellipsoid. For $p>1$, equality holds if and only if $K$ is an origin-symmetric ellipsoid. Naturally, a new collection of sharp affine Sobolev inequality accompanies \cite{HLPRY23_2}: for $p\in [1,n)$, $m\in \mathbb{N}$, a convex body $Q$ in $\R^m$ containing the origin, and $f\in W^{1,p}(\R^n)$, we have
\begin{equation}
\label{eq:m_lp_affine_sobolev}
a_{n,p}\|f\|_{L^{\frac{np}{n-p}}(\Rn)}\leq d_{n,p}(Q)\left(\int_{\S} \left( \int_{\Rn} h_Q((\nabla f(z))^t\theta)^pdz \right)^{-\frac{nm} p } d\theta\right)^{-\frac 1 {nm}}:=\mathcal{E}_{p}(Q,f)
\end{equation}
where
\begin{equation}
\label{eq:sharp}
d_{n,p}(Q) := (n\omega_n)^\frac{1}{p}\left(nm\Vol(\PP \B )\right)^\frac 1 {nm},
\end{equation}
and
\begin{equation} \label{eq:sobolev_cons}
a_{n,p}=n^\frac{1}{p}\left(\frac{n-p}{p-1}\right)^\frac{p-1}{p}\left[\frac{\omega_n}{\Gamma(n)}\Gamma\left(\frac{n}{p}\right)\Gamma\left(n+1-\frac{n}{p}\right)\right]^\frac{1}{n}, \quad a_{n,1}=\lim_{p\to 1^+}a_{n,p},
\end{equation}
 
 We remark that the constant $d_{n,p}(Q)$ is so that the following comparison is valid.

\begin{theorem}
\label{t:relate}
	Let $p\geq 1,m\in\N,$ and $f\in W^{1,p}(\Rn)$. Then, for every convex body $Q$ in $\mathbb{R}^m$ containing the origin, one has
	\begin{equation}
 \label{eq:relate:}
		\begin{aligned}
			\mathcal{E}_{p}(Q,f) \leq \|\nabla f\|_{L^p(\Rn)},
		\end{aligned}
	\end{equation}
 with equality when $f$ is radially symmetric. For $p=1$, the inequality can be extended to $BV(\R^n)$. In this case, $\|\nabla f\|_{L^1(\Rn)}$ should be understood as $|Df|(\Rn)$, the total variation of $f$.
\end{theorem}

An immediate consequence of Theorem \ref{t:relate} is that \eqref{eq:m_lp_affine_sobolev} is stronger than the classical $L^p$ Sobolev inequality \eqref{eq:lp_sobolev}.  Here, note that the ``multiplicity'' $m$ is hidden implicitly in the dimension of $Q$ for the notation $\mathcal{E}_{p}(Q,f)$. We remark that the energy $\mathcal{E}_{p}(Q,f)$ can be viewed as a higher-order affine version of the Dirichlet energy $\|\nabla f\|_{L^p(\Rn)}$, as well as the $L^p$ affine energy $\mathcal{E}_p(f)$. In particular, when $m=1,Q=[-1/2,1/2]$, we have $\mathcal{E}_p([-1/2,1/2], f)= \mathcal{E}_p(f)$; when $m=1, Q=[0,1]$, this recovers the asymmetric $L^p$ affine energy defined implicitly in Haberl-Schuster \cite{HS2009}. A consequence of Theorem~\ref{t:relate} is that $$\mathcal{E}_{p}(Q,f^\star) = \|\nabla f^\star\|_{L^p(\Rn)},$$ regardless of the choice of $m\in\N$ and convex body $Q\subset\R^m$ containing the origin.

A natural question is: is there a {P}\'olya-{S}zeg\"o principle for the higher-order affine energy $\mathcal{E}_{p}(Q,f)$? We answer this question positively.

\begin{theorem}
\label{t:mps}
  Let $p\geq 1,m\in\N$, $f\in W^{1,p}(\Rn)$, and a convex body $Q$ in $\R^m$ containing the origin. Then,
    $$\mathcal{E}_p(Q,f)\geq \mathcal{E}_p(Q,f^\star).$$
    When $p>1$ and $f$ satisfies the minor regularity assumption 
    \begin{equation}\vol\left(\left\{x:\left|\nabla f^{\star}(x)\right|=0\right\} \cap\left\{x: 0<f^{\star}(x)<\|f\|_{L^{\infty}(\Rn)}\right\}\right)=0,
\label{eq:regularity}
\end{equation}
    there is equality if and only if the level sets of $f$ are dilations of an ellipsoid in $\Rn$ with respect to its center.
  \end{theorem}
  
  Note that Theorem \ref{t:mps} is valid for all $p\geq 1$ whereas \eqref{eq:m_lp_affine_sobolev} is only true for $p\in [1,n)$.

We remark here that the regularity assumption \eqref{eq:regularity} for the equality condition cannot be removed due to a classical counterexample constructed by Brothers and Ziemer \cite{BZ88}. In the $m=1$ case, Theorem \ref{t:mps} is due to Lutwak, Yang, and Zhang \cite{LYZ02}, Cianchi, Lutwak, Yang, and Zhang \cite{CLYZ09}, Haberl, Schuster, and Xiao \cite{HSX12}, and Nguyen \cite{NVH16}. The extension to the space of $BV(\Rn)$ in this case is due to Wang \cite{TW12, TW13}. Our methods of proof, therefore, are inevitably influenced by them.

Affine {P}\'olya-{S}zeg\"o principles can be used to ``upgrade'' classical (non-affine) isoperimetric inequalities. Indeed, Theorem \ref{t:relate}, \ref{t:mps}, in combination with  the classical $L^p$ Sobolev inequality \eqref{eq:lp_sobolev}, imply the $m$th-order affine Sobolev inequality \eqref{eq:m_lp_affine_sobolev} established in \cite{HLPRY23_2}. In fact,
\begin{equation}\label{eq 8.19.1}
	\mathcal{E}_p(Q,f)\geq \mathcal{E}_p(Q,f^\star)= \|\nabla f^\star\|_{L^p(\Rn)}\geq a_{n,p}\|f^\star\|_{L^{\frac{np}{n-p}}(\Rn)} = a_{n,p}\|f\|_{L^{\frac{np}{n-p}}(\Rn)}.
\end{equation}

The same philosophy will be exploited in Section \ref{sec:app} to demonstrate a variety of affine versions of classical (non-affine) Sobolev-type inequalities.

Finally, in Section~\ref{sec:poin}, we consider an associated Poincar\'e-type inequality for $\mathcal{E}_p(Q,f)$ when $Q$ is origin-symmetric.
\begin{theorem}\label{t:Poincare} Let $\Omega \subset \Rn$ be a bounded domain containing the origin in its interior, $m\in\N$, $Q\in\conbodio[m]$ be origin-symmetric, and $1 \leq p < \infty$. Then, there exists a constant $C:=C(n,m,p,Q, \Omega) >0$ such that, for any function $f \in W^{1,p}_0(\Omega)$ it holds that 
\[
\mathcal{E}_p(Q,f) \geq C \|f\|_{L^p(\Rn)}. 
\]  
\end{theorem} 

\section{Preliminaries}
\label{sec:back}

We first recall some rudimentary facts about convex bodies; here and throughout this section, the book \cite{Sh1} by Schneider is the standard reference. We denote by $\conbod[d]\supset \conbodo[d] \supset \conbodio[d]$, respectively, the collection of convex bodies in $\R^d$, those that additionally contain the origin, and those that contain the origin in their interiors. We use $\mathbb{S}^{d-1}$ for the centered unit sphere in $\R^d$.

For each $K\in \conbodio[d]$, its \emph{polar body} $K^\circ\in \conbodio[d]$ is given by 
$$K^\circ = \bigcap_{x\in K}\{y\in \R^n: y^tx\leq 1\}.$$
The \emph{support function} $h_K:\mathbb{S}^{d-1}\rightarrow \mathbb{R}$ of a compact, convex set $K\subset \R^d$ is given by
\begin{equation}
	h_K(u)=\sup_{y\in K}y^tu, \qquad\text{ for every } u \in \mathbb{S}^{d-1}.
\end{equation}
Note that the support function can be extended to $\R^d$ by making it 1-homogeneous. The \emph{Minkowski functional} of $K\in \conbodio[d]$ is defined to be 
\begin{equation}
	\|y\|_K=\inf\{r>0:y\in rK\}.
\end{equation}
Note that $\|\cdot\|_{K}$ is the possibly asymmetric norm whose unit ball is $K$. The Minkowski functional, support function, and polar body are related via the following equation:
\begin{equation}
	h_K(u) = \|u\|_{K^\circ}.
\end{equation}
By polar coordinates, the volume of $K\in \conbodio[d]$ can be expressed in the following way through its Minkowski functional:
\begin{equation}
\label{eq:polar_coordinates}
    \vol[d](K)=\frac{1}{d}\int_{\s}\|\theta\|_K^{-d}d\theta.
\end{equation}

The \emph{surface area measure}, $\sigma_K$, for a $K \in \conbod[d]$ is the Borel measure on $\mathbb{S}^d$ given by
\[\sigma_K(D)=\mathcal H ^{d-1}(n^{-1}_K(D)),\]
for every Borel subset $D$ of $\s$, where $n_K^{-1}(D)$ consists of all boundary points of $K$ with outer unit normals in $D$. 

 Suppose the boundary $\partial K$ of a compact set $K\subset \R^d$ has enough regularity, e.g. $K$ is convex or $C^1$ smooth. The \emph{mixed volume} of $K$ with a compact, convex set $L$ is given by
\begin{equation}
\label{eq:mixed_smooth}
V_{1}(K,L)=\frac{1}{d}\int_{\partial K}h_L(n_K(y))d\mathcal{H}^{d-1}(y).
\end{equation}
\emph{Minkowski's first inequality} (see, for example, \cite{GZ99}) states:
\begin{equation}
  V_{1}(K,L)^d\geq \vol(K)^{d-1}\vol(L),
  \label{eq:Firey_Min_first}
\end{equation}
with equality if and only if $K$ is homothetic to $L$.

When $K\in \conbodio[d]$, Lutwak \cite{LE93, LE96} introduced the \emph{$L^p$ surface area measure}, denoted by $\sigma_{K,p}$, of $K$,
\begin{equation}
	d\sigma_{K,p}(u)=h_K(u)^{1-p}d\sigma_{K}(u). 
\end{equation}
The $L^p$ surface area measure is arguably one of the cornerstones of the $L^p$ Brunn-Minkowski theory that occupies much of the research in the theory of convex bodies in the last few decades.

We now collect some notations and basic facts about a measurable function $f \colon \R^d \to \R$. Throughout the rest of the section, we fix $p\geq 1$. 

A function $f$ is said to be \emph{non-constant} if, for every $\alpha \in \R$, $\vol(\{v \in \Rn \colon f(x) \neq \alpha \}) >0$. A function $f \colon \R^d \to \R$ is said to have a \emph{weak derivative} if $\nabla f$ exists in the weak sense; that is, there exists a measurable vector map $\nabla f:\Rn\to\Rn$ such that
\begin{equation}\int_{\R^d} f(v) \text{ div} \psi(v) dv = - \int_{\R^d} (\nabla f(v))^t \psi(v) dv
\label{eq_LYZBorel}
 \end{equation}
for every compactly supported, smooth vector field $\psi:\R^d \to \R^d$ (see \cite{Evans}). We say a function is an $L^p$ Sobolev function, and belongs to the $L^p$ Sobolev space $W^{1,p}(\R^d)$, if $f,\nabla f \in L^p(\R^d)$.  The space of functions of bounded variation, denoted by $BV(\R^d)$, is very similar, except in the right-hand side of \eqref{eq_LYZBorel}, the vector map $\nabla f$ is replaced with a more generic map $\sigma_f:\R^d\mapsto\R^d$, and integration with respect to the Lebesgue measure is replaced by integration with respect to $|Df|$, the total variation measure of $f$. The two spaces are related: if $f$ is in $BV(\R^d)$ and $f$ has a weak derivative, then, $d|Df|(v)=|\nabla f(v)| dv$ and $\sigma_f(v)=\nabla f(v)/|\nabla f(v)|$ for $|Df|$ almost all $v\in \R^d$.

The space of compactly supported, infinitely-times differentiable functions on $\R^d$ will be denoted by $C_0^\infty(\R^d)$.

\emph{Federer's coarea formula} (see e.g. \cite[Page 258]{FH69}) states: if $f:\R^d\to\R$ is Lipschitz and $g:\R^d\to [0,\infty)$ is measurable, then, for any Borel set $A\subseteq \R$, one has
\begin{equation}
\label{coarea}
    \int_{f^{-1}(A) \cap\{|\nabla f|>0\}} g(x) d x=\int_A \int_{f^{-1}(t)} \frac{g(y)}{|\nabla f(y)|} d \mathcal{H}^{d-1}(y) d t .
\end{equation}
Here, $\mathcal{H}^{d-1}$ denotes the $(d-1)$-dimensional Hausdorff measure. 

The \emph{distribution function} of a measurable function $f:\R^d\rightarrow\R$ is given by 
\begin{equation}
\label{eq:distribution_function}
\mu_f(t)=\vol[d]\left(\left\{x\in\R^d: |f(x)| >t\right\}\right),\end{equation}
and is finite for $t>0$ if $f\in L^p(\R^d)$. Its \emph{decreasing rearrangement} $f^\ast:[0,\infty)\rightarrow [0,\infty]$ is defined by
\begin{equation}
  f^\ast(s)=\sup\{t>0:\mu_f(t)>s\} \quad \text{for }s\geq 0.
\end{equation}
The \emph{spherically symmetric rearrangement} $f^\star$ of $f$ is defined as
\begin{equation}
\label{eq:rearrange}
    f^\star(x)=f^\ast(\omega_d|x|^d) \quad \text{for } x\in \R^d.
\end{equation}
Spherically symmetric rearrangement preserves the distribution function, that is
\begin{equation}
    \mu_{f}=\mu_{f^\star}.
\label{eq:dis_funs}
\end{equation}
Moreover, for every continuous, increasing function $\Phi:[0,\infty)\mapsto[0,\infty)$, one has
\begin{equation}
    \int_{\Rn}\Phi(|f(x)|)dx = \int_{\Rn}\Phi(f^\star(x))dx.
    \label{eq:change_of_dimension}
\end{equation}
Similarly, one may define the rearrangement of $f$ with respect to any $K\in\conbodio[d]$; see, \emph{for example}, \cite{AFTL97}. That is,
 \begin{equation}
\label{eq:consym}
    f^K(x)=f^\ast(\vol[d](K)\|x\|^d_{K}) \quad \text{for } x\in\R^d.
\end{equation}
Note that when $K$ is the centered unit ball, it recovers $f^\star$.

Let $\nu$ be a finite Borel measure on $\s$ that is not concentrated in any closed hemisphere, $p\geq 1,m\in\N$, and $Q\in \conbodo[m]$. Define the $(L^p,Q)$ cosine transform of $\nu$, denoted by $C_{p,Q}^m \mu :\S\rightarrow (0,\infty)$ as 
\begin{equation}
	(C_{p,Q}^m \mu) (\theta) = \int_{\s} h_Q(v^t\theta)^pd\mu(v).
\end{equation}
Note that since $\nu$ is not concentrated in any closed hemisphere, the function $C_{p,Q}^m\nu$ is a positive, continuous function on $\S$. Moreover, since $h_Q$ is convex and $p\geq 1$, when extended $1$-homogeneously to $\R^{nm}$, the function $C_{p,Q}^m \mu$ is convex.

We shall use the following definition that slightly generalizes the $(L^p,Q)$ projection body given in Definition \ref{d:generalprojectionbody}.

\begin{definition}
\label{def:mu_bodies}
Let $\nu$ be a finite Borel measure on $\s$ that is not concentrated in any closed hemisphere, $p\geq 1,m\in\N$, and $Q\in \conbodo[m]$. The $(L^p,Q)$ projection body of $\nu$ is the convex body in $\conbodio[nm]$  whose support function is given by
\begin{equation}
	h_{\P \nu} = \left(C_{p,Q}^m \nu\right)^\frac{1}{p}. 
\end{equation}
\end{definition}
We then have $\PP \nu = (\P \nu)^\circ$. When $\nu = \sigma_{K,p}$, the $L^p$ surface area measure of $K\in \conbodio[n]$ one has $\PP \sigma_{K,p} = \PP K.$ We also require the following definition.
\begin{definition}\label{d:generalLYZbody_2} For $p \geq 1$,  $m \in \N$, and $Q \in \conbodo[m]$, define the $(L^p,Q)$-LYZ projection body of non-constant $f\in W^{1,p}(\Rn)$, $\LYZ f$ to be the convex body in $\R^{nm}$ whose support function is given, for every $\theta\in \S$,  by 
\begin{equation}
\label{eq:LYZ_body_2}
h_{\LYZ f}(\theta) = \left(\int_{\R^n}h_Q((\nabla f(x))^t\theta )^p dx\right)^\frac{1}{p}=\|h_Q(\nabla f(\cdot)^t\theta)\|_{L^p(\Rn)}.
\end{equation}
\end{definition}
It will be shown in \eqref{eq 8.15.3} that $\LYZ f$ in fact contains the origin in its interior and consequently, its polar body $\LYZP f:=(\LYZ f)^\circ$ is well-defined.

We remark here the link among $\LYZ f$, $\P \nu$, and $\P K$.  Lutwak, Yang, and Zhang \cite{LYZ06} showed  that there exists a unique measure on the sphere $\nu_{f,p}$ that is not concentrated on any great hemisphere, so that the integration over $\R^n$ in \eqref{eq:LYZ_body_2} can replaced with integration over $\s$ with respect to $\nu_{f,p}$ ($\nabla f(x)$ becomes merely $u\in \s$). Historically, $Q=[-\frac{1}{2},\frac{1}{2}]$, and thus one can apply the even $L^p$ Minkowski problem by Lutwak \cite{LE93} to the even part of $\nu_{f,p}$ to obtain an unique centrally symmetric convex body whose $L^p$ surface area measure is the even part of $\nu_{f,p}$. This origin-symmetric convex body associated to $f$ via the even part of $\nu_{f,p}$, called by Ludwig the \textit{LYZ body of $f$} \cite{LUD}, became a powerful tool in convex geometry; for example, Wang \cite{TW13} used the LYZ body to characterize, independently of Nguyen, equality conditions for the {P}\'olya-{S}zeg\"o principle for the functional $\mathcal{E}_p(f)$. This approach was later extended to another energy functional over the Grassmanian by Kniefacz and Schuster \cite{KS21}.

In our situation, we would follow \cite{HLPRY23_2} and apply, for example, the non-even $L^p$ Minkowski problem by Hug, Lutwak, Yang and Zhang \cite[Theorem 1.3]{HLYZ05} to the measure $\mu_{f,p}$ to obtain unique convex bodies $\langle f \rangle_p\in\conbodo[n]$, the so-called asymmetric LYZ bodies. But, the origin may be on the boundary of $\langle f \rangle_p$ for $1<p < n$. Consequently, $\LYZ f$ is not necessarily $\P$ applied to  $\langle f \rangle_p$, and, in particular, the Petty projection inequality \eqref{eq 8.25.8} does not readily apply. Such an issue is not present for $p> n$ or $p=1$.

The following equation is immediate following Definition \ref{d:generalLYZbody_2} and the definition of $\mathcal{E}_p(Q,f)$:
\begin{equation}
    \mathcal{E}_p(Q,f) = d_{n,p}(Q)(nm)^{-\frac{1}{nm}}\Vol(\LYZP f)^{-\frac{1}{nm}}.
    \label{eq:relation_E}
\end{equation}

\section{An $m$th-order affine P\'olya-Szeg\"o principle}
\label{sec:main}

This section is dedicated to proving Theorems \ref{t:relate} and \ref{t:mps}, the $m$th-order affine P\'olya-Szeg\"o principle. We shall see in later sections that this is the source of many functional \emph{sharp} affine isoperimetric inequalities. The approach adopted here is inspired by \cite{CLYZ09,  HS2009,HSX12, LYZ00, GZ99}.
\subsection{Proving Theorem~\ref{t:relate}}
A bit of simple matrix manipulation is needed. For each $T\in GL(n)$, consider the map $\overline{T}\in GL(nm)$ given as $\overline T(x_1, \cdots, x_m) = (Tx_1,\dots,  Tx_m)$ for any $x_1,\dots, x_m\in \R^n$. It is simple to see that if $T\in O(n)$, then $\overline{T}\in O(nm)$. In particular, if $\theta \in \S$, then $\overline T \theta\in \S$. Write $\theta = (\theta_1, \dots, \theta_m)$ where $\theta_i\in \mathbb{R}^n$. Then $\overline T \theta = (T\theta_1, \dots, T\theta_n)$ and consequently, for every $v\in \mathbb{R}^n$, we have
\begin{equation}\label{eq 8.15.10}
	v^t \overline{T}\theta = (v^t T\theta_1, \cdots, v^tT\theta_n)= ((T^tv)^t\theta_1, \dots, (T^tv)^t\theta_n)= (T^tv)^t\theta.
\end{equation}

\begin{proof}[Proof of Theorem~\ref{t:relate}]
	When integrating on $O(n)$ with respect to $dT$, we mean integrating with respect to the Haar measure on $O(n$). We start by deriving a new formula for $d_{n,p}(Q)$. Note that the integral
	\begin{equation}
        \label{eq:step 1.5}
		\int_{O(n)} h_Q((T^tv)^t\theta)^pdT
	\end{equation}
	is independent of $v\in \s$. This and the definitions of $dT$ and $\PP \B$ imply
	\begin{equation}
 \label{eq:step 0}
		\begin{aligned}
            d_{n,p}(Q)^{nm} &= (n\omega_n)^{\frac{nm}{p}} nm \Vol(\PP \B)
            =  (n\omega_n)^{\frac{nm}{p}}\int_{\S}\|\theta\|_{\PP \B}^{-nm}d\theta
            \\
            &= \int_{\S}\left(\frac{1}{n\omega_n}\int_{\s} h_Q(v^t\theta)^pdv \right)^{-\frac{nm} p }d\theta
            \\
			&=\int_{\S}\left(\int_{O(n)} h_Q((T^t e_1)^t\theta)^pdT \right)^{-\frac{nm} p }d\theta.
		\end{aligned}
	\end{equation}
	Here $e_1\in \s$ is, say, the first vector in the canonical basis of $\R^n$. Henceforth, we work with $W^{1,p}(\Rn)$; the proof still works with minor modifications in the $p=1,f\in BV(\Rn)$ case.
	
	Recall that for every $T\in O(n)$, we have $\overline{T}\in O(nm)$. Thus, a change-of variable formula and \eqref{eq 8.15.10}  imply that for every $T\in O(n)$, we have, 
	\begin{equation}\label{eq:step 3}
		\begin{aligned}
			\left(d_{n,p}(Q)^{-1}\mathcal{E}_{p}(Q,f)\right)^{-nm}&=\int_{\S}\|h_Q(\nabla f(\cdot)^t\theta)\|_{L^p(\Rn)}^{-nm} d\theta\\
        = \int_{\S}\|h_Q(\nabla f(\cdot)^t\overline{T}\theta)\|_{L^p(\Rn)}^{-nm} d\theta&= \int_{\S}\|h_Q((T^t\nabla f(\cdot))^t\theta)\|_{L^p(\Rn)}^{-nm} d\theta.
      	\end{aligned}
	\end{equation}
	Consequently, by the fact that $dT$ is a probability measure on $O(n)$, the Fubini theorem and Jensen's inequality yield
 \begin{equation}\label{eq 8.18.1}
		\begin{aligned}
			\left(d_{n,p}(Q)^{-1}\mathcal{E}_{p}(Q,f)\right)^{-nm}
			= &\int_{O(n)}\int_{\S}\|h_Q((T^t\nabla f(\cdot))^t\theta)\|_{L^p(\Rn)}^{-nm} d\theta dT\\
		=&\int_{\S} \int_{O(n)}\|h_Q((T^t\nabla f(\cdot))^t\theta)\|_{L^p(\Rn)}^{-nm} dT d\theta \\
		\geq & \int_{\S} \left(\int_{O(n)}\|h_Q((T^t\nabla f(\cdot))^t\theta)\|_{L^p(\Rn)}^pdT\right)^{-\frac{nm}{p}} d\theta.
		\end{aligned}
	\end{equation} 
 Next, we insert the definition of $\|h_Q((T^t\nabla f(\cdot))^t\theta)\|_{L^p(\Rn)}$. We then use the homogeneity of support functions, the Fubini theorem once again, the fact that \eqref{eq:step 1.5} is independent of $v\in \s$, and finally \eqref{eq:step 0} to obtain from \eqref{eq 8.18.1}
	\begin{equation}\label{eq 8.18.1_2}
		\begin{aligned}
			&\left(d_{n,p}(Q)^{-1}\mathcal{E}_{p}(Q,f)\right)^{-nm}\\
		\geq & \int_{\S} \left(\int_{O(n)}\int_{\R^n} |\nabla f(z)|^ph_Q\left(\left(T^t \frac{\nabla f(z)}{|\nabla f(z)|}\right)^t\theta\right)^pdzdT\right)^{-\frac{nm}{p}} d\theta\\
		=& \int_{\S} \left(\int_{\R^n} |\nabla f(z)|^p\int_{O(n)}h_Q\left(\left(T^t \frac{\nabla f(z)}{|\nabla f(z)|}\right)^t\theta\right)^pdTdz\right)^{-\frac{nm}{p}} d\theta\\
		= &\|\nabla f\|_{L^p{(\R^n)}}^{-nm}\int_{\S} \left(\int_{O(n)}h_Q\left(\left(T^t e_1\right)^t\theta\right)^pdT\right)^{-\frac{nm}{p}} d\theta\\
		= & \|\nabla f\|_{L^p{(\R^n)}}^{-nm} d_{n,p}(Q)^{nm}.
		\end{aligned}
	\end{equation} 
	This yields the desired inequality. When $f$ is radially symmetric, for every $T\in O(n)$, by a change-of-variable and then the chain rule, we have,
	\begin{equation}
		\begin{aligned}
			\|h_Q((T^t\nabla f(\cdot))^t\theta)\|_{L^p(\Rn)}^p &= \int_{\R^n} h_Q((T^t\nabla f(z))^t\theta)^pdz=  \int_{\R^n} h_Q((T^t\nabla f(Tz))^t\theta)^pdz\\
			&= \int_{\R^n} h_Q((\nabla f\circ T(z))^t\theta)^pdz= \int_{\R^n} h_Q((\nabla f(z))^t\theta)^pdz \\
       &= \|h_Q((\nabla f(\cdot))^t\theta)\|_{L^p(\Rn)}^p.
		\end{aligned}
	\end{equation}

	Thus, equality holds in \eqref{eq 8.18.1} by the equality condition of Jensen's inequality.
		
\end{proof}

\subsection{Proving the inequality in Theorem \ref{t:mps}}
Throughout this section, let $p\in [1,\infty)$ be fixed. We first prove an inequality for the polar projection bodies $\PP \nu$ introduced in Definition~\ref{def:mu_bodies}.

\begin{lemma}\label{l:body}
	Let $\nu $ be a finite Borel measure on $\s$ that is not concentrated in any closed hemisphere. Then, there exists $K\in \conbodo[n]$ (dependent on $\nu$ and $p$) such that 
    \begin{equation}
		1 = \frac{1}{n} \int_{\s} h_{K}(u)^p d\nu(u)
	\end{equation}
    and
    \begin{equation}\vol(K)h_{K}(u)^{p-1}d\mu(u)=d\sigma_{K}(u).
    \label{eq:PDE_final}
    \end{equation}
 Moreover, the convex body $K$ and the Borel measure $\nu$ satisfy, for every $m\in\N$ and $Q\in \conbodo[m]$,
	\begin{equation}\label{eq 8.16.1}
		\Vol(\PP \nu ) \vol(K)^{-m} \leq \Vol(\PP \B ) \vol(\B)^{\frac{nm} p -m}.
	\end{equation}
    
\end{lemma}
\begin{proof}
	Since $\nu$ is not concentrated in any closed hemisphere, we may  construct a sequence of discrete measures $\nu_i$ such that $\nu_i$ are not concentrated in any closed hemisphere and $\nu_i$ converges to $\nu$ weakly. See, for example, \cite[Pg 392-393]{Sh1}. 
	
	By the solution to the discrete $L^p$ Minkowski problem (see, \emph{e.g.},  \cite[Theorem 1.1]{HLYZ05}), there exist polytopes $P_i\in \conbodio[n]$ such that 
    \begin{equation}
        \vol(P_i)h_{P_i}(u)^{p-1}d\nu_i(u)=d\sigma_{P_i}(u).
        \label{eq:PDE}
    \end{equation}
    Since each $P_i\in \conbodio[n]$, this means
	\begin{equation}\label{eq 8.13.2}
		\nu_i = \vol(P_i)^{-1}\sigma_{P_i, p}.
	\end{equation} 
	An application of \cite[Lemmas 2.2, 2.3]{HLYZ05} shows that $P_i$ is uniformly bounded. Blaschke's selection theorem \cite[Theorem 1.8.6]{Sh1} now implies that, by passing to a subsequence, we may assume the $P_i$ converge, in the Hausdorff metric, to some $K\in \conbodo[n]$. 
 
 By \eqref{eq 8.13.2}, we have
	\begin{equation}
		1 = \frac{1}{n} \int_{\s} h_{P_i}(u)^p \frac{d\sigma_{P_i, p}(u)}{\vol(P_i)} = \frac{1}{n} \int_{\s} h_{P_i}(u)^p d\nu_i(u) \rightarrow \frac{1}{n} \int_{\s} h_{K}^p(u) d\nu(u), 
	\end{equation}
	where the convergence follows from the fact that $h_{P_i}\rightarrow h_{K}$ uniformly (since $h_{P_i}\rightarrow h_K$ uniformly and $p\geq 1$) and that $\nu_i$ converges to $\nu$ weakly. Taking the limit in \eqref{eq:PDE} yields \eqref{eq:PDE_final}.
	
	Since each $P_i$ contains the origin as an interior point, \eqref{eq 8.25.8} applies:
	\begin{equation}\label{eq 8.13.3}
		\Vol(\PP P_i ) \vol(P_i)^{\frac{nm} p -m} \leq \Vol(\PP \B ) \vol(\B)^{\frac{nm} p -m}.
	\end{equation}
	The weak convergence of $\nu_i$ to $\nu$ yields
 $$\left(C_{p,Q}^m \left(\nu_i\right)\right)^\frac{1}{p} \rightarrow \left(C_{p,Q}^m (\nu)\right)^\frac{1}{p}$$ pointwise on $\S$ and, consequently, uniformly on $\S$. By definition, this means that
 $$\PP \nu_i \rightarrow \PP \nu$$ in the Hausdorff metric. By taking volume, we obtain $\Vol(\PP \nu_i)\rightarrow \Vol(\PP \nu)$. On the other hand, from \eqref{eq 8.13.2}, we have that 
 $$\PP \nu_i = \vol(P_i)^\frac{1}{p}\PP P_i.$$
 Taking volume, this yields
	\begin{equation}
		\vol(P_i)^\frac{nm}{p}\Vol(\PP P_i)\rightarrow \Vol(\PP \nu).
	\end{equation}	
 Letting $i\rightarrow \infty$ in \eqref{eq 8.13.3} completes the proof.
\end{proof}
We remark that, if $K \in\conbodio[n]$, then $\sigma_{K,p}$ exists and, \eqref{eq:PDE_final} implies, 
$$\nu = \vol(K)^{-1}\sigma_{K,p}.$$
Thus, $\PP \nu = \vol(K)^\frac{1}{p} \PP K$. Consequently, equation \eqref{eq:PDE_final} simply states the usual $m$th-order Petty projection inequality \eqref{eq 8.25.8} for this $K$. As another example, if $\nu=\mu_{f,p}$, the measure described above \eqref{eq:relation_E}, then the resultant $K$ in Lemma~\ref{l:body} is precisely $\langle f \rangle_p$.

The following lemma is critical in reducing Theorem~\ref{t:mps} from $W^{1,p}(\Rn)$ to $C_0^\infty(\Rn)$. We also obtain some useful facts about $\LYZP f$ from Definition~\ref{d:generalLYZbody_2}.
\begin{lemma}\label{lemma 8.15.1}
	 Let $p\geq 1,m\in\N$, $f_k, f\in W^{1,p}(\Rn)$, and $Q\in\conbodo[m]$. If $f_k\rightarrow f$ in $W^{1,p}(\Rn)$, then 
	\begin{equation}
		\|\theta\|_{\LYZP {f_k}}=\| h_Q(\nabla f_k(\cdot)^t\theta)\|_{L^p(\Rn)}\rightarrow \| h_Q(\nabla f(\cdot)^t\theta)\|_{L^p(\Rn)}=\|\theta\|_{\LYZP f}
	\end{equation}
	uniformly for $\theta\in \S$, as $k\rightarrow \infty$, i.e. $\LYZP {f_k}\to \LYZP f$ in the Hausdorff metric. Moreover, if $f$ is not constantly $0$ (up to a set of measure $0$), then there exists $c_0>0$ such that for every $\theta\in \S$
	\begin{equation}\label{eq 8.15.3}
		\|\theta\|_{\LYZP f}=\| h_Q(\nabla f(\cdot)^t\theta)\|_{L^p(\Rn)}>c_0,
	\end{equation}
and consequently $\LYZP f \in \conbodio[nm]$.
\end{lemma}
\begin{proof}
	Since $Q\in \conbodo[m]$, there exist $M>0$ such that $Q\subset MB_2^m$. In particular, using the definition of support function, we get that for all $x,y\in \R^m$,
	\begin{equation}\label{eq 8.15.1}
		|h_Q(x)-h_Q(y)|\leq M|x-y|.
	\end{equation}
Since $\S$ is compact, there exists $\lambda>0$ such that 
	\begin{equation}\label{eq 8.15.2}
		|x^t\theta|\leq \lambda |x|, \qquad \text{for all } x\in \R^n \text{ and all } \theta\in \S.
	\end{equation}
	
	Using \eqref{eq 8.15.1} and \eqref{eq 8.15.2}, if $p\geq 1$ and $f\in W^{1,p}(\Rn)$,
	\begin{equation}
		\begin{aligned}
			\|h_Q(\nabla f_k(\cdot)^t\theta) - h_Q(\nabla f(\cdot)^t\theta)\|_{L^p(\Rn)}&\leq  M\|\nabla (f_k-f)(\cdot)^t\theta) \|_{L^p(\Rn)}\\
			&\leq M\lambda \|\nabla (f_k-f)\|_{L^{p}(\Rn)}\\
			&\rightarrow 0,
		\end{aligned}
	\end{equation}
	uniformly in $\theta\in \S$.
	
	It remains to show \eqref{eq 8.15.3}. Since $Q\in \conbodo[m]$, there exist linearly independent $u_1, \cdots, u_m\in \mathbb{S}^{m-1}$ and $\gamma>0$ such that $\gamma u_i\in Q$. This implies,
	\begin{equation}\label{eq 8.15.4}
		h_Q(x)\geq \gamma \max\{(x^tu_1)_+, \dots, (x^tu_m)_+\}.
	\end{equation}
	Note that for every $\theta\in \S$, we have
	\begin{equation}
		\max\{|\theta u_1|, \dots, |\theta u_m|\}>0,
	\end{equation}
	since $u_1, \dots, u_m$ are linearly independent and $\theta$ is not the zero map. This, when combined with the fact that continuous functions on a compact set achieve minima, implies the existence of $\tau>0$ such that for all $\theta\in \S$, 
	\begin{equation}\label{eq 8.15.7}
		\max\{|\theta u_1|, \dots, |\theta u_m|\}>\tau.
	\end{equation}
	In particular, this means there exists $u_{i_*}$ (for each $\theta$) such that $|\theta u_{i_*}|>\tau$.
	
	 It is simple to see the map
	\begin{equation}
		\theta\mapsto \| h_Q(\nabla f(\cdot)^t\theta)\|_{L^p(\Rn)}
	\end{equation}
	is continuous on $\S$. Moreover, by \eqref{eq 8.15.4} and \eqref{eq 8.15.7},
	\begin{equation}
	\begin{aligned}
		\| h_Q(\nabla f(\cdot)^t\theta)\|_{L^p(\Rn)}&\geq \gamma \max_{1\leq i\leq m}\|(\nabla f(\cdot)^t\theta u_i)_+\|_{L^p(\Rn)}\\
		& = \gamma |\theta u_{i_*}|\left\|\left(\nabla f(\cdot)^t\frac{\theta u_{i_*}}{|\theta u_{i_*}|}\right)_+\right\|_{L^p(\Rn)}\\
		&\geq \gamma \tau\left\|\left(\nabla f(\cdot)^t\frac{\theta u_{i_*}}{|\theta u_{i_*}|}\right)_+\right\|_{L^p(\Rn)}\\
		&>0.
	\end{aligned}		
	\end{equation} 
	Here, the last line follows from the first part of the proof in \cite[Lemma 2]{HS2009}. Consequently, the compactness of $\S$ implies the existence of $c_0>0$ such that \eqref{eq 8.15.3} holds.
\end{proof}

We are now ready to apply the geometric inequality, Lemma \ref{l:body}, to prove Theorem \ref{t:mps}. Let $f\in C_0^\infty (\Rn)$. Then, by Sard's theorem, one has for a.e. $t>0$ that $$[f]_t:=\{|f|\geq t\}$$ is a bounded, closed set with $C^1$ boundary $$\partial [f]_t=\{|f|=t\}.$$ 

\begin{proof}[Proof of Theorem~\ref{t:mps}, the inequality]
We assume $f$ is not constantly 0 (up to a set of measure 0), since otherwise, the desired inequality is trivial. 

We will prove the case when $f\in C_0^\infty(\Rn)$.

Recall that by Sard's theorem, for almost all $t\in (0, \|f\|_{\infty})$ we have $\nabla f(y)\neq o$ for $y\in \partial [f]_t$. For those $t$, consider the finite Borel measure $\nu_t$ on $\s$ such that for all $g\in C(\s)$, we have
\begin{equation}
\label{eq:convexifying_measure}
	\int_{\s} g(v)d\nu_t(v) = \int_{\partial [f]_t} g\left(\frac{\nabla f(y)}{|\nabla f (y)|} \right) |\nabla f(y)|^{p-1} d\mathcal{H}^{n-1}(y).
\end{equation}
Note that $\nu_t$ is not concentrated in any closed hemisphere. This can be seen by taking $g(v) = (v^tu)_+$ for an arbitarily fixed $u\in \s$ and noting that the integral is positive, since $f\in C_0^\infty(\R^n)$ and the outer unit normals of $\partial [f]_t$ form $\s$. In fact, on $\partial[f]_t$, one has that 
$$n(y):=n_{\partial[f]_t}(y)=\frac{\nabla f(y)}{|\nabla f (y)|}=\sigma_f(y).$$
By Lemma~\ref{l:body}, there exists a convex body, which we denote by $\langle f \rangle_{t,p}\in \conbodo[n]$, such that
\begin{equation} \label{eq 8.13.5}
	\Vol(\PP \nu_t ) \leq \Vol(\PP \B ) \vol(\B)^{\frac{nm} p -m} \vol(\langle f \rangle_{t,p})^{m}
\end{equation}
and
\begin{equation}
	1 = \frac{1}{n} \int_{\s} h_{\langle f \rangle_{t,p}}(v)^p d\nu_t.
\end{equation}

 Observe that, from the coarea formula \eqref{coarea}, Minkowski's integral inequality, and \eqref{eq 8.13.5}:
 \begin{equation}\label{eq:our_split_affine_sobolev}
	\begin{aligned}
		  &d_{n,p}(Q)^{-p} \mathcal{E}_p^p(Q, f) \\
    =& \left(\int_{\S} \left( \int_0^{\infty} \int_{\partial[f]_t} h_{Q}^p\left(n(y)^t\theta\right) |\nabla f(y)|^{p-1}d\mathcal{H}^{n-1}(y)dt \right)^{-\frac{nm}{p}}d\theta \right)^{-\frac{p}{nm}}\\
  =& \left(\int_{\S} \left( \int_0^{\infty} \int_{\s} h_{Q}^p\left(v^t\theta\right) d\nu_t(v)dt \right)^{-\frac{nm}{p}} d\theta\right)^{-\frac{p}{nm}}\\
  =& \left(\int_{\S} \left( \int_0^{\infty} \|\theta\|_{\PP \nu_t}^p dt \right)^{-\frac{nm}{p}} d\theta\right)^{-\frac{p}{nm}}\\
  \geq& \int_0^{\infty} \left(\int_{\S}  \|\theta\|_{\PP \nu_t}^{-nm} d\theta \right)^{-\frac{p}{nm}}dt\\
  =& (nm)^{-\frac{p}{nm}}\int_0^{\infty} \Vol(\PP \nu_t)^{-\frac{p}{nm}} dt\\
  \geq& (nm\vol(\PP \B))^{-\frac{p}{nm}}\omega_n^{\frac{p-n}{n}} \int_{0}^{\infty} \vol(\langle f \rangle_{t,p})^{-\frac{p}{n}}dt.
	\end{aligned}
\end{equation}

Inserting the definition of  $d_{n,p}(Q)$ (see \eqref{eq:sharp}), we obtain
\begin{equation}
\label{eq:zeroth_bound}
\mathcal{E}_p(Q,f) \geq  n^\frac{1}{p}\omega_n^{\frac{1}{n}} \left(\int_0^\infty \vol(\langle f \rangle_{t,p})^{-\frac{p}{n}}dt\right)^\frac{1}{p}.
\end{equation}
Suppose $p=1$. Then, one obtains from \eqref{eq:mixed_smooth} and \eqref{eq:Firey_Min_first}
\begin{align*}
    1 &= \frac{1}{n} \int_{\s} h_{\langle f \rangle_{t,1}} d\nu_t = \frac{1}{n}\int_{\partial [f]_t} h_{\langle f \rangle_{t,1}}(n(y)) d\mathcal{H}^{n-1}(y)
    \\
    &= V_{1}([f]_t,\langle f \rangle_{t,1}) \geq \vol([f]_t)^\frac{n-1}{n}\vol(\langle f \rangle_{t,1})^\frac{1}{n}
    \\
    &= \mu_f(t)^\frac{n-1}{n}\vol(\langle f \rangle_{t,1})^\frac{1}{n},
\end{align*}
where the last equality follows from \eqref{eq:distribution_function}.  Similarly, if $p>1$, we have from H\"older's inequality and again \eqref{eq:mixed_smooth} and \eqref{eq:Firey_Min_first}
\begin{align*}
    &n^\frac{1}{p}\left(\int_{\partial [f]_t}|\nabla f(y)|^{-1}d\mathcal{H}^{n-1}(y)\right)^\frac{p-1}{p}
    \\
    &= \left(\int_{\partial [f]_t}h_{\langle f \rangle_{t,p}}(n(y))^p|\nabla f(y)|^{p-1}d\mathcal{H}^{n-1}(y)\right)^\frac{1}{p}\left(\int_{\partial [f]_t}|\nabla f(y)|^{-1}d\mathcal{H}^{n-1}(y)\right)^\frac{p-1}{p}
    \\
    &\geq \int_{\partial [f]_t}h_{\langle f \rangle_{t,p}}(n(y))d\mathcal{H}^{n-1}(y)
    \geq n\vol([f]_t)^\frac{n-1}{n}\vol(\langle f \rangle_{t,p})^\frac{1}{n}
    \\
    &= n\mu_f(t)^\frac{n-1}{n}\vol(\langle f \rangle_{t,p})^\frac{1}{n}.
\end{align*}
In either case, we deduce that for $p\geq 1$
\begin{equation}
    \vol(\langle f \rangle_{t,p})^{-\frac{p}{n}} \geq \mu_f(t)^{\frac{p}{n}(n-1)}\left(\frac{1}{n}\int_{\partial [f]_t}|\nabla f(y)|^{-1}d\mathcal{H}^{n-1}(y)\right)^{1-p}.
    \label{eq:first_bound}
\end{equation}
Notice there is equality if and only if $\langle f \rangle_{t,p}$ is homothetic to $[f]_t$ for a.e. $t>0$. Inserting this into \eqref{eq:zeroth_bound}, we obtain

\begin{equation}
    \label{eq:second_bound}
    \mathcal{E}_p(Q,f) \geq  n^\frac{1}{p}\omega_n^{\frac{1}{n}} \left(\int_0^\infty \mu_f(t)^{\frac{p}{n}(n-1)}\left(\frac{1}{n}\int_{\partial [f]_t}|\nabla f(y)|^{-1}d\mathcal{H}^{n-1}(y)\right)^{1-p}dt\right)^\frac{1}{p}.
\end{equation}

Next, from another use of the coarea formula \eqref{coarea}, we have (see e.g. \cite{BZ88,CLYZ09})
$$\mu_f(t) = \vol\left([f]_t\cap \{\nabla f =0\}\right)+\int_t^\infty\int_{\{|f|=\tau\}}|\nabla f(y)|^{-1}d\mathcal{H}^{n-1}(y)d\tau.$$
Notice both terms on the right-hand side are nonincreasing. Therefore, we may differentiate in the variable $t$ and throw-away the first term to deduce
\begin{equation}
    -\mu^\prime_f(t) \geq \int_{\partial [f]_t}|\nabla f(y)|^{-1}d\mathcal{H}^{n-1}(y), \quad \text{for a.e. } t>0.
    \label{eq:bounds_distribution_deriv}
\end{equation}
One has, see e.g. \cite[Lemmas 2.4 and 2.6]{CF02}, that equality holds in \eqref{eq:bounds_distribution_deriv} if and only if $f=f^\star$.

Inserting \eqref{eq:bounds_distribution_deriv} into \eqref{eq:second_bound} yields
\begin{equation}
    \label{eq:fourth_bound}
    \mathcal{E}_p(Q,f) \geq n\omega_n^{\frac{1}{n}} \left(\int_0^\infty \frac{\mu_f(t)^{\frac{p}{n}(n-1)}}{(-\mu^\prime_f(t))^{p-1}}dt\right)^\frac{1}{p}.
\end{equation}
By concatenating all the equality cases mentioned above, equality in \eqref{eq:fourth_bound}  holds if $f$ is replaced by $f^*$. By the fact that $\mu_f(t) = \mu_{f^*}(t)$, this completes the proof in the case of $f\in C_0^\infty$.

The general case then follows from approximation. This is essentially the same as \cite[Theorem 2.1]{CLYZ09}, but we recall the approximation for the convenience of the reader. Given $f\in W^{1,p}(\Rn),$ let $\{f_k\}$ be a sequence of $C_0^\infty(\Rn)$ functions that converge to $f$ in $W^{1,p}(\Rn)$. By the first part of the proof, for all $k=1,2,\dots$,
\begin{equation}
\label{eq:smooth_inequality_approx}
\mathcal{E}_p(Q,f_k^*) \leq \mathcal{E}_p(Q,f_k)\end{equation}
Lemma~\ref{lemma 8.15.1} and the definition of $\mathcal{E}_p(Q, \cdot)$ show that 
\begin{equation}
\label{eq:smooth_limit}
\lim_{k\to\infty}\mathcal{E}_p(Q,f_k)=\mathcal{E}_p(Q,f).\end{equation}

On the other hand, since rearrangements are contractive in $L^p(\Rn)$ (see \cite{CG79}), we know $f_k^*\rightarrow f^*$ in $L^p(\Rn)$ and consequently $f_k^\star\to f^\star$ weakly in $W^{1,p}(\Rn)$. This, when combined with Theorem~\ref{t:relate} and the fact that the $L^p$ gradient norm is lower-semicontinuous with respect to weak convergence in $W^{1,p}(\Rn)$, shows
\begin{equation}
\begin{split}
	\mathcal{E}_p(Q,f)=\lim_{k\to\infty}\mathcal{E}_p(Q,f_k)&\geq \liminf_{k\to\infty}\mathcal{E}_p(Q,f_k^*)= \liminf_{k\to\infty}\|\nabla f_k^*\|_{L^p(\Rn)}
 \\
 &\geq \|\nabla f^*\|_{L^p(\Rn)}=\mathcal{E}_p(Q,f^*).
 \end{split}
\end{equation}

We save the equality conditions for Section~\ref{sec:equality} below.

\end{proof}
\begin{remark}
Following \cite{CLYZ09,LXZ11,LYZ02,GZ99}, we call the sets the $\langle f \rangle_{t,p}$ the \textit{asymmetric convexification of $[f]_t$}.  We isolate from \eqref{eq:our_split_affine_sobolev} the following inequality:

\begin{align*}
   \mathcal{E}_p^p(Q, f) & > n\omega_n\Vol(\PP \B )^\frac p {nm} \int_0^{\infty} \Vol(\PP \nu_t)^{-\frac{p}{nm}} dt\\
  &\geq n\omega_n^\frac{p}{n} \int_{0}^{\infty} \vol(\langle f \rangle_{t,p})^{-\frac{p}{n}}dt.
\end{align*}
This may be viewed as the higher-order analog of the convex Lorentz-Sobolev inequality found in \cite{LXZ11,XJ07}.
\end{remark}

\subsection{The equality conditions in Theorem~\ref{t:mps}}
\label{sec:equality}
This section is dedicated to proving the following statement:
     Let $p> 1$, $f\in W^{1,p}(\Rn)$, and $Q\in\conbodo[m]$. Suppose $f$ satisfies the minor regularity condition \eqref{eq:regularity}. Then, $\mathcal{E}_p(Q,f) = \mathcal{E}_p(Q,f^*)$ if and only if $f(x)=f^E(x+x_0)$ for some $x_0\in\Rn$ and origin-symmetric ellipsoid $E\in \conbodio[n]$.

Our approach is heavily inspired by \cite{NVH16}. Our first step is to introduce some notation.
\begin{definition}
  \label{d:generalcentroidbody}
  Let $p \geq 1$, $m \in \N$, and fix some $Q \in \conbodo[m]$. Given a compact set $L\subset \R^{nm}$ with positive volume, we define the $(L^p, Q)$-centroid body of $L$, $\G L$, to be the convex body in $\Rn$ with the support function 
  \begin{equation}
  h_{\G L}(v)^p= \frac 1 {\Vol(L)}\int_L h_Q(v^tx)^p dx.
  \label{eq:cen_hi}
\end{equation}
\end{definition}
This of course extends on the classical centroid body operators, which we will discuss in more detail further below. We now derive a different formula for $\mathcal{E}_p(Q,f)$.

\begin{lemma}
\label{l:yet_another_formula}
Let $p\geq 1,m\in\N$, $f\in W^{1,p}(\Rn)$ not identically zero, and $Q\in\conbodo[m]$. Then,
$$\mathcal{E}_p(Q,f) = d_{n,p}(Q) \left(\Vol(\LYZP f)(nm+p) \int_{\R^n}h_{\G \LYZP f}(\nabla f(x))^p dx\right)^{-\frac{1}{nm}}.$$
\end{lemma}
\begin{proof}
    By definition, we have from Fubini's theorem and polar coordinates
    {\allowdisplaybreaks\begin{align*}
        &\int_{\R^n}h_{\G \LYZP f}(\nabla f(x))^p dx \\
        =& \frac 1 {\Vol(\LYZP f)}\int_{\LYZP f} \int_{\R^n} h_Q((\nabla f(y)^tx)^p dy dx
        \\
        =&\frac 1 {\Vol(\LYZP f)}\int_{\LYZP f} \|x\|^p_{\LYZP f} dx
        \\
        =&\frac 1 {\Vol(\LYZP f)} \int_{\S} \|\theta\|^p_{\LYZP f} \int_0^{\|\theta\|^{-1}_{\LYZP f}}r^{nm-1+p}drd\theta
        \\
        =&\frac 1 {\Vol(\LYZP f)(nm+p)} \int_{\S} \|\theta\|^{-nm}_{\LYZP f} d\theta
        \\
        =&\frac 1 {\Vol(\LYZP f)(nm+p)} \int_{\S}\left(\int_{\R^n}h_Q((\nabla f(x))^t\theta )^p dx\right)^{-\frac{nm}{p}}d\theta
        \\
        =&\frac{\mathcal{E}_p(Q,f)^{-nm}d_{n,p}(Q)^{nm}}{\Vol(\LYZP f)(nm+p)}.
    \end{align*}}
\end{proof}

Our next step is to replace $\Vol(\LYZP f)$ with $\vol(\G\LYZP f)$. We will need the following from \cite[Theorem 1.6]{HLPRY23_2}. We suppress the definition of star body or compact domain; simply note that this applies to elements of $\conbodio[nm]$.
\begin{theorem}[The $m$th-order $L^p$ Busemann-Petty inequality]
    \label{t:NewLpBPC} Fix $n,m \in \N$. 
Let $L\subset \R^{nm}$ be a compact set with positive volume, $Q \in \conbodo[m]$, and $p \geq 1$. Then
    \begin{equation}\label{eq:NewLpBPC}
     \frac{\vol(\G L)}{\Vol(L)^{1/m}} \geq \frac{\vol(\G \PP \B)}{\Vol(\PP \B)^{1/m}}.
\end{equation}
If $L$ is a star body or a compact set with piecewise smooth boundary, then there is equality if and only if $L = \PP E$ up to a set of zero volume for some origin-symmetric ellipsoid $E \in \conbodio[n]$.
\end{theorem}
This of course extends the known $m=1$ cases: the classical $Q=[-\frac{1}{2},\frac 12]$ results from Busemann and Petty \cite{Busemann53,petty61_1} when $p=1$ and from Lutwak, Yang and Zhang \cite{LYZ00} when $p>1.$ By setting $Q = [-\alpha_1, \alpha_2]$, $\alpha_1,\alpha_2>0$ we obtain the asymmetric $L^p$ case by Haberl and Schuster \cite{HS09}. 

Setting $L=\LYZP f$ in \eqref{eq:NewLpBPC}, one obtains
\begin{equation}
\label{eq:LPBC_applied}
\Vol(\LYZP f)^{1/m} \leq \frac{\Vol(\PP \B)^{1/m}}{\vol(\G \PP \B)}\vol(\G \LYZP f),\end{equation}
with equality if and only if $\LYZP f=\PP E$ for some origin-symmetric ellipsoid $E\in\conbodio[n]$. To work with these equality conditions we use \cite[Lemma 3.12]{HLPRY23_2}: $\G \circ \PP$ is bijective on the class of centered ellipsoids, in the sense that
\begin{equation}
\label{eq:elipp_cal}
\G \PP E = \omega_n^{\frac{1}{p}}\vol(E)^{-\frac 1 {p}} \left(\frac m {\omega_n(nm+p)}\right)^{\frac 1p } E.\end{equation}
Combining this with Lemma~\ref{l:yet_another_formula}, we obtain the following.

\begin{proposition}
\label{prop:BP_energy}
Let $p\geq 1,m\in\N$, $f\in W^{1,p}(\Rn)$ not identically zero, and $Q\in\conbodo[m]$. Then,
    $$\mathcal{E}_p(Q,f) \geq \left(\frac{\omega_n}{\vol(\G \LYZP f)}\right)^\frac{1}{n}\left(\int_{\R^n}h_{\G \LYZP f}(\nabla f(x))^p dx\right)^{\frac{1}{p}}.$$
    with equality if and only if $\G \LYZP f = E$ for some origin-symmetric ellipsoid $E\in\conbodio[n]$.
\end{proposition}
\begin{proof}
    Observe that, by using, Lemma~\ref{l:yet_another_formula}, \eqref{eq:relation_E}, \eqref{eq:LPBC_applied}, and \eqref{eq:sharp},
    {\allowdisplaybreaks\begin{align}
        &\mathcal{E}_p(Q,f) = \mathcal{E}_p(Q,f)^\frac{nm+p}{p} \mathcal{E}_p(Q,f)^{-\frac{nm}{p}}
        \\
        & = \mathcal{E}_p(Q,f)^\frac{nm+p}{p} d_{n,p}(Q)^{-\frac{nm}{p}} \left(\Vol(\LYZP f)(nm+p) \int_{\R^n}h_{\G \LYZP f}(\nabla f(x))^p dx\right)^{\frac{1}{p}}
        \\
        &=d_{n,p}(Q)(nm)^{-\frac{nm+p}{nmp}}(nm+p)^\frac{1}{p}\vol(\LYZP f)^{-\frac{1}{nm}}\left(\int_{\R^n}h_{\G \LYZP f}(\nabla f(x))^p dx\right)^{\frac{1}{p}}
        \\
        &\geq \left(\omega_n\frac{nm+p}{m}\right)^\frac{1}{p}\left(\frac{\vol(\G \PP \B)}{\vol(\G \LYZP f)}\right)^\frac{1}{n}\left(\int_{\R^n}h_{\G \LYZP f}(\nabla f(x))^p dx\right)^{\frac{1}{p}}.
    \end{align}}
    Using \eqref{eq:elipp_cal} completes the proof.
\end{proof}

We have the following two facts. The first can be seen as a P\'olya-Szeg\"o principle for convex symmetrization.
\begin{lemma}[Theorem 1.1 in \cite{FV04}]
\label{l:PS_convex}
    Fix $p>1$ and origin-symmetric $K\in\conbodio[n]$. For any $f\in W^{1,p}(\Rn)$, one has
    \begin{equation}
        \int_{\Rn}h_K(\nabla f(x))^p dx \geq \int_{\Rn}h_{K}(\nabla f^K(x))^p dx.
        \label{eq:convex_PS}
    \end{equation}
    Moreover, if $f$ is nonnegative and satisfies the minor regularity assumption of \begin{equation}\vol\left(\left\{x:\left|\nabla f^{K}(x)\right|=0\right\} \cap\left\{x: 0<f^{K}(x)<\|f\|_{L^\infty(\R^n)}\right\}\right)=0,
\label{eq:regularity_2}
\end{equation}
then equality holds in \eqref{eq:convex_PS} if and only if $f(x)=f^K(x+x_0)$ for some $x_0\in\Rn$.
\end{lemma}
\begin{proposition}
\label{p:equal_form}
    Let $p > 1$, $m\in\N,$ $f\in W^{1,p}(\Rn)$ not identically zero, and $Q\in\conbodo[m]$. Fix any $K\in\conbodio[n]$ such that $\vol(K)=\omega_n.$ Then, one has
    $$\mathcal{E}_p(Q,f^\star) = \left(\int_{\R^n}h_{K}(\nabla^K f(x))^p dx\right)^{\frac{1}{p}}.$$
\end{proposition}
\begin{proof}
    We use from Theorem~\ref{t:relate} that $\mathcal{E}_p(Q,f^\star)=\|\nabla f^*\|_{L^p(\Rn)}=\mathcal{E}_p([-1/2,1/2],f^\star)$, and then the claim follows from \cite[Lemma 3.5]{NVH16}.
\end{proof}
The equality characterization for Theorem~\ref{t:mps} now follows: we set $$L=\left(\frac{\omega_n}{\vol(\G \LYZP f)}\right)^\frac{1}{n}\G \LYZP f.$$ Then, $\vol(L)=\omega_n$ and from Proposition~\ref{prop:BP_energy} we have $ \mathcal{E}_p(Q,f) \geq \left(\int_{\R^n}h_{L}(\nabla f(x))^p dx\right)^{\frac{1}{p}},$
    with equality only when $L$ is an origin-symmetric ellipsoid $E$. Finally, when $L=E$, use Lemma~\ref{l:PS_convex} and Proposition \ref{p:equal_form} to finish the equality characterizations.

\section{Applications of the affine P\'olya-Szeg\"o principle}
\label{sec:app}

The central theme of this section is to demonstrate a zoo of affine ``upgrades'' of well-known functional isoperimetric inequalities using the affine P\'olya-Szeg\"o principle (Theorem~\ref{t:mps}). The key ingredient is that in the case of radially symmetric functions, the affine energy $\mathcal{E}_p(Q, f)$ reduces to its Euclidean counterpart $\|\nabla f\|_{L^p(\Rn)}$. This philosophy was demonstrated in \eqref{eq 8.19.1} when it was shown that Theorem~\ref{t:mps}, Theorem~\ref{t:relate}, and the classical $L^p$ Sobolev inequality \eqref{eq:lp_sobolev}, imply the $m$th-order $L^p$ affine Sobolev inequality \eqref{eq:m_lp_affine_sobolev}.

The same philosophy applies to a wide range of well-known functional inequalities. We shall skip the proofs of most of the theorems to-be-presented, as they are more or less just along the lines of \eqref{eq 8.19.1}. Occasionally, a proof will be given if additional details are needed. We remark that these affine inequalities will imply their classical (nonaffine) counterparts, thanks to Theorem~\ref{t:relate}.

Note that the $L^p$ Sobolev inequality \eqref{eq:lp_sobolev} and the $m$th-order affine Sobolev inequality \eqref{eq:m_lp_affine_sobolev} are only valid for $p\in [1,n)$. For $p>n$, the Morrey-Sobolev embedding theorem states that any compactly supported function in $W^{1,p}(\Rn)$ is essentially bounded. This can be written as a sharp inequality, see, \emph{e.g.}, Talenti \cite[Theorem 2.E]{GT94}:
 \begin{equation}\|f\|_{L^\infty}(\Rn) \leq b_{n,p}\vol(\text{supp}(f))^{\frac{1}{p}\left(\frac{p}{n}-1\right)}\|\nabla f\|_{L^p(\Rn)},
\label{eq:MS}
\end{equation}
for every $f\in W^{1,p}(\Rn)$ with support $\text{supp}(f)$ having finite volume. Here 
\begin{equation}
\label{eq:constant_2}
b_{n,p}=n^{-\frac{1}{p}}\omega_n^{-\frac{1}{n}}\left(\frac{p-1}{p-n}\right)^\frac{p-1}{p}.\end{equation}
There is equality for functions of the form 
\begin{equation}
\label{eq:f_ellip}
f_{MS}(x)=a(1-|x-x_0|^\frac{p-n}{p-1})_+
\end{equation}
for some $a\in \R$ and $x_0\in\Rn$.

Chaining Theorems \ref{t:relate}, \ref{t:mps}, and \eqref{eq:MS} as in \eqref{eq 8.19.1} provides the following affine version of \eqref{eq:MS}. 
\begin{theorem}[The $m$th-order $L^p$ affine Morrey-Sobolev inequality]
\label{t:MS}
    Fix $p>n \geq 1$ and $f\in W^{1,p}(\Rn)$ such that $\text{supp}(f)$ has finite volume. Then, for every $m\in\N$ and $Q\in\conbodo[m]$
    \begin{equation}\|f\|_{L^\infty(\Rn)} \leq b_{n,p}\vol(\text{supp}(f))^{\frac{1}{p}\left(\frac{p}{n}-1\right)}\mathcal{E}_p(Q,f).
\label{eq:our_affine_MS}
\end{equation}
Here, the sharp constant $b_{n,p}$ is given by \eqref{eq:constant_2}. There is equality for functions of the form $f_{MS}\circ A$, where $f_{MS}$ is given by \eqref{eq:f_ellip} and $A\in \GL(n)$.
\end{theorem}
This recovers the $m=1$ cases given by \cite[Theorem 1.2]{CLYZ09} when $Q=[-\frac 12,\frac 12]$   and \cite[Corollary 9]{HSX12} when  $Q=[0,1]$ .

It was shown in \cite[Proposition 4.6]{HLPRY23_2} that $\lim_{p\rightarrow \infty} d_{n,p}(Q)$ exists. We shall denote the limit as $d_{n,\infty}(Q)$. This motivates the following definition.
\begin{definition}
\label{def:affine_linfty_m_energy}
    Let $f\in W^{1,\infty}(\Rn), m\in\N,$ and let $Q\in \conbodo[m]$. Then, the $(L^\infty,Q)$ affine Sobolev energy of $f$ is given by 
    \begin{equation*}
\mathcal{E}_{\infty}(Q,f)=d_{n,\infty}(Q)\left(\int_{\S} \|h_Q((\nabla f(\cdot))^t\theta)\|_{L^\infty(\Rn)}^{-nm} d\theta\right)^{-\frac 1 {nm}}.
\end{equation*}
\end{definition}
Taking the limit in \eqref{eq:our_affine_MS} shows the following Faber-Krahn-type inequality for $\mathcal{E}_{\infty}(Q,f)$. This recovers the $m=1, Q=[0,1]$ case by Haberl, Schuster and Xiao \cite[Corollary 10]{HSX12}.

\begin{corollary}[Faber-Krahn Inequality for $\mathcal{E}_{\infty}(Q,f)$]\label{corollary 8.19.1}
    Let $m\in\N$, $Q\in \conbodo[m]$ and $f\in W^{1,\infty}(\Rn)$ such that $\text{supp}(f)$ has finite volume. Then,    \begin{equation}
    \|f\|_{L^\infty}(\Rn) \leq \omega_n^{-1/n}\vol(\text{supp}(f))^{\frac{1}{n}}\mathcal{E}_\infty(Q,f).
\label{eq:our_FK}
\end{equation}
Equality holds if  $f$ is of the form
\begin{equation}\label{eq 8.19.2}
	f(x)=a(1-|A(x-x_0)|)_+
\end{equation}
for some $a\in \R,x_0\in\Rn$ and $A\in \GL(n)$. 
\end{corollary}
\begin{proof}
   Note that the fact that equality holds in \eqref{eq:our_FK} for $f$ in \eqref{eq 8.19.2} follows from a direct computation. The inequality follows from Theorem~\ref{t:MS} by taking a limit. Indeed, from Fatou's lemma (and the fact that $d_{n,p}(Q)$ is continuous in $p$),
    \begin{align*}
\|f\|_{L^\infty(\Rn)} & \leq \omega_n^{-1 / n} \vol(\text{supp}(f))^{1 / n} \limsup _{q \rightarrow \infty} \mathcal{E}_q(Q,f) \\
& \leq \omega_n^{-1 / n} \vol(\text{supp}(f))^{1 / n}d_{n,\infty}(Q)\left(\liminf _{q \rightarrow \infty} \int_{\S}\left\|h_Q((\nabla f(\cdot))^t\theta)\right\|_{L^q(\Rn)}^{-nm} d \theta\right)^{-\frac{1}{nm}} \\
& \leq \omega_n^{-1 / n} \vol(\text{supp}(f))^{1 / n}d_{n,\infty}(Q)\left(\int_{\S} \liminf _{q \rightarrow \infty}\left\|h_Q((\nabla f(\cdot))^t\theta)\right\|_{L^q(\Rn)}^{-nm} d \theta\right)^{-\frac{1}{nm}} \\
& \leq \omega_n^{-1 / n} \vol(\text{supp}(f))^{1 / n}d_{n,\infty}(Q)\left(\int_{\S}\left\|h_Q((\nabla f(\cdot))^t\theta)\right\|_{L^\infty(\Rn)}^{-nm} d \theta\right)^{-\frac{1}{nm}} .
\end{align*}
\end{proof}

Let $n>1$. Denote
\begin{equation}
    m_n = \sup_{\phi}\int_0^\infty e^{\phi(t)^\frac{n}{n-1}-t}dt,
    \label{eq:constant}
\end{equation}
where the supremum is taken over all non-decreasing locally absolutely continuous functions in $[0,\infty)$ such that $\phi(0)=0$ and $\int_0^\infty \phi^\prime(t)^ndt \leq 1$. The Moser-Trudinger inequality \cite{MJ70,TN67} states that for every function $f\in W^{1,n}(\Rn)$ with $0<\vol(\text{supp}(f))<\infty$, we have
\begin{equation}
    \label{eq:mt}
    \frac{1}{\vol(\text{supp}(f))} \int_{\Rn}e^{\left(n\omega_n^{1/n}\frac{|f(x)|}{\|\nabla f\|_{L^n(\Rn)}}\right)^\frac{n}{n-1}}dx \leq  m_n.\end{equation}
  The constant $n\omega_n^\frac{1}{n}$ is  best possible in the sense that if it were to be replaced by any other larger number, the above inequality would fail for some $f\in W^{1,n}(\Rn)$ with $0<\vol(\text{supp}(f))<\infty$. Carleson and Chang \cite{CC86} showed that spherically symmetric extremals do exist for \eqref{eq:mt}. Theorems \ref{t:relate} and \ref{t:mps} now imply the following affine version; when $m=1$, we recover the $Q=[-\frac 12,\frac 12]$ case from \cite[Theorem 1.1]{CLYZ09} and the $Q=[0,1]$ case from \cite[Corollary 8]{HSX12}.
 
\begin{theorem}[The $m$th-order $L^p$ affine Moser-Trudinger inequality]
\label{t:MT}
    Let $n>1$, $m\in\N$, and $Q\in\conbodo[m]$. Then, for every $f\in W^{1,n}(\Rn)$ with $0<\vol(\text{supp}(f)) < \infty$, we have
    \begin{equation}
    \label{eq:our_mt}
    \frac{1}{\vol(\text{supp}(f))} \int_{\Rn}e^{\left(n\omega_n^{1/n}\frac{|f(x)|}{\mathcal{E}_n(Q,f)}\right)^\frac{n}{n-1}}dx \leq  m_n.\end{equation}
    The constant $n\omega_n^{1/n}$ is best possible in the sense that if it were to be replaced by any other larger number, the above inequality would fail for some $f\in W^{1,n}(\Rn)$ with $0<\vol(\text{supp}(f))<\infty$.
\end{theorem}

Recall that the sharp logarithmic Sobolev inequality shown in \cite{FW78} states: for $f\in W^{1,2}(\Rn)$ such that $\|f\|_{L^2(\R^n)}=1$, one has
$$\int_{\Rn}|f|^2\log |f|dx \leq \frac{n}{2}\log\left(\left(\frac{2}{ne\pi}\right)^\frac{1}{2}\|\nabla f\|_{L^2(\R^n)}\right).$$
Ledoux \cite{ML96} and Del Pino and Dolbeault \cite{DD03} established the following extension of the log-Sobolev inequality: for $1\leq p <n$ and $f\in W^{1,p}(\Rn)$  such that $\|f\|_{L^p(\Rn)}=1,$ one has
\begin{equation}
    \int_{\Rn}|f|^p\log |f| dx \leq \frac{n}{p}\log(c_{n,p}\|\nabla f\|_{L^p(\Rn)}),
    \label{eq:lp_log}
\end{equation}
where \begin{equation}
\label{eq:sharp_constant_3}
c_{n,p} =\left(\frac{p}{n}\right)^{1 / p}\left(\frac{p-1}{e}\right)^{1-1 / p}\left(\frac{\Gamma\left(1+\frac{n}{2}\right)}{\pi^{n / 2} \Gamma\left(1+\frac{n(p-1)}{p}\right)}\right)^{1 / n}, \quad c_{n,1}=\lim_{p\to 1^+} c_{n,p}.
\end{equation} As for equality conditions, when $p=1$, Beckner \cite{BW99} showed that one must look beyond $W^{1,1}(\Rn)$ to $BV(\Rn),$ where there is equality only for dilates of characteristic functions of centered Euclidean balls. Carlen \cite{CE91}, for $p=2$, and Del Pino and Dolbeault \cite{DD03} for $1<p<n$, showed there is equality if and only if there exists $a>0$ and $x\in \Rn$ such that
\begin{equation}
\label{eq:log_sobo_eq}
f_{LS}(x)=\frac{\pi^{n / 2} \Gamma\left(1+\frac{n}{2}\right)}{a^{n(p-1) / p} \Gamma\left(1+\frac{n(p-1)}{p}\right)} \exp \left(-\frac{1}{a}\left|x-x_0\right|^{p /(p-1)}\right).\end{equation}

The following corollary includes, as special case \cite[Corollary 2]{HSX12} ($m=1, Q=[0,1]$). It follows from Theorems \ref{t:relate}, \ref{t:mps}, and \eqref{eq:lp_log}.

\begin{theorem}[The $m$th-order $L^p$ affine log-Sobolev inequality]
    Let $1\leq p <n$ and $f\in W^{1,p}(\Rn)$. Then, for every $m\in\N$ and $Q\in\conbodo[m]$, one has
    $$\int_{\Rn}|f|^p\log |f| dx \leq \frac{n}{p}\log(c_{n,p}\mathcal{E}_p(Q,f)).$$
    Here, the sharp constant $c_{n,p}$ is given by \eqref{eq:lp_log}. If $p=1$, there is equality for dilates of characteristic functions of centered ellipsoids. On the other-hand, If $1<p<n$, then there is equality for functions of the form $f_{LS}\circ A$, were $f_{LS}$ is of the form \eqref{eq:log_sobo_eq} and $A\in \GL(n)$.
\end{theorem}

Nash's inequality, shown in its optimal form by Carlen and Loss \cite{CL93}, states: for $f\in L^1(\Rn)\cap W^{1,2}(\Rn),$ one has
\begin{equation}
\|f\|_{L^2(\Rn)}\left(\frac{\|f\|_{L^2(\Rn)}}{\|f\|_{L^1(\Rn)}}\right)^\frac{2}{n}\leq \beta_n\|\nabla f\|_{L^2(\Rn)},
    \label{eq:nash}
\end{equation}
where
\begin{equation}
    \beta_n^2=\frac{2\left(1+\frac{n}{2}\right)^{1+n / 2}}{n \lambda_n \omega_n^{2 / n}},
    \label{eq:nash_constant}
\end{equation}
and $\lambda_n$ is the first nonzero Neumann eigenvalue of the Laplacian $-\Delta$ on radial functions on $B_2^n$. In fact, let $u$ be the associated eigenfunction. Then, there is equality in \eqref{eq:nash_constant} if and only if $f$ is a normalized and scaled version of 
\begin{equation}
    f_N(x)=\begin{cases}
        u(|x-x_0|) - u(1), &\text{if }|x| \leq 1,
        \\
        0, & \text{if }|x| \geq 1,
    \end{cases}
    \label{eq:function_nash}
\end{equation}
for some $x_0\in \Rn$. 

Applying Theorems \ref{t:relate}, \ref{t:mps}, and \eqref{eq:nash}, we immediately obtain the following theorem, which extends the case $m=1,Q=[0,1]$ by Haberl, Schuster and Xiao \cite[Corollary 11]{HSX12}.
\begin{theorem}[The $m$th-order affine $L^p$ Nash inequality]
    If $f\in L^1(\Rn)\cap W^{1,2}(\Rn),$ then, for every $m\in\N$ and $Q\in\conbodo[m]$, one has
    $$\|f\|_{L^2(\Rn)}\left(\frac{\|f\|_{L^2(\Rn)}}{\|f\|_{L^1(\Rn)}}\right)^\frac{2}{n}\leq \beta_n\mathcal{E}_p(Q,f),$$
    where $\beta_n$ is given by \eqref{eq:nash_constant}. 
    There is equality for functions of the form $f_N\circ A$, where $f_N$ is of the form \eqref{eq:function_nash} and $A\in \GL(n)$.
\end{theorem}

We recall that the Gagliardo-Nirenberg inequalities are precisely
\begin{equation}
\|f\|_{L^r(\Rn)}\alpha_{n,p}(r,q)\leq \|\nabla f\|_{L^p(\Rn)}^\theta\|f\|_{L^q(\Rn)}^{1-\theta}
\label{eq:GN_ineq}
\end{equation}
where $1<p<n, q<r\leq \frac{np}{n-p}$, $\theta\in (0,1)$ is so that the inequality is scale-invariant. Note that inequality \eqref{eq:GN_ineq} can be deduced from the $L^p$ Sobolev inequality \eqref{eq:lp_sobolev}, but without the optimal constant $\alpha_{n,p}(r,q)$; this is still open in general. Clearly, \eqref{eq:lp_sobolev} and Nash's inequality \eqref{eq:nash} are special cases. In fact, the logarithmic Sobolev inequality \eqref{eq:lp_log} is also a limiting case. 

Del Pino and Dolbeault \cite{DD02,DD03} at the turn of the century made a breakthrough in establishing sharp constants for a range of parameters in \eqref{eq:GN_ineq}. This was followed by \cite{CENV04}, where mass transport was used to give an elegant proof of the same results, among other things. Suppose again that $1<p<n$ and now pick $q\in (p,\frac{pn}{n-p}-\frac{p}{n-p}]$. Then, set
\begin{equation}
\label{eq:params}
r=\frac{p(q-1)}{p-1} \quad \text{and} \quad \theta=\frac{n(q-p)}{(q-1)(np-(n-p)q)}.\end{equation}
Let $W_0^{1,p,q}(\Rn)$ denote the completion of the space of smooth compactly supported functions with respect to the norm given by
$$\|f\|_{p,q}=\|\nabla f\|_{L^p(\Rn)} + \|f\|_{L^q(\Rn)}.$$
Then, it was shown \cite{DD02,DD03,CENV04} that \eqref{eq:GN_ineq} holds for every $f\in W_0^{1,p,q}(\Rn)$ with constant
\begin{equation}
    \label{eq:sharp_constant_4}
    \alpha_{n,p}(r,q)=\left(\frac{p q}{\delta}\right)^{\frac 1 r}\left(\left(\frac{p \sqrt{\pi}}{q-p}\right)\left(\frac{n(q-p)}{p q}\right)^{\frac 1 p}\left(\frac{\Gamma\left(\frac{\delta(p-1)}{p(q-p)}\right) \Gamma\left(1+\frac{n(p-1)}{p}\right)}{\Gamma\left(\frac{q(p-1)}{q-p}\right) \Gamma\left(1+\frac{n}{2}\right)}\right)^{\frac1 n}\right)^\theta,
\end{equation}
where $\delta=np-q(n-p)$. Among the class $W_0^{1,p,q}(\Rn)$, equality holds if and only if $f$ is of the form
\begin{equation}
\label{GN_eq}
    f_{GN}(x)=a\left(1+b|x-x_0|^\frac{p}{p-1}\right)^{-\frac{p-1}{q-p}}
\end{equation}
for some $a\in \R, b>0$ and $x_0\in\Rn$. When restricted to the class $W_0^{1,p,q}(\Rn)$ (with the aforementioned choices of $q$ and $r$), \eqref{eq:GN_ineq} interpolates between the sharp $L^p$ Sobolev inequality \eqref{eq:lp_sobolev} and the logarithmic Sobolev inequality \eqref{eq:lp_log}. Indeed, the former is when $q=p(n-1)/(n-p)$ (which yields $t=1$ and $\alpha_{n,p}(r,q)=a_{n,p}$ from \eqref{eq:sobolev_cons}), and the latter follows by sending $q\to p$ from above.

Applying, like before, Theorems \ref{t:relate}, \ref{t:mps}, and  \eqref{eq:GN_ineq}, we immediately obtain the following, which extends on the case $m=1, Q=[0,1]$ by Haberl, Schuster and Xiao \cite[Corollary 12]{HSX12}. This formally interpolates between \eqref{eq:lp_sobolev} and \eqref{eq:lp_log}.
\begin{theorem}
    Let $1<p<n,p<q\leq \frac{p(n-1)}{(n-p)}, m\in\N$ and let $r$ and $\theta$ be given by \eqref{eq:params}. If $f\in W_0^{1,p,q}(\Rn)\cap W^{1,p}(\Rn)$, then, for every $Q\in\conbodo[m]$, one has
    \begin{equation}
\|f\|_{L^r(\Rn)}\alpha_{n,p}(r,q)\leq \mathcal{E}_p(Q,f)^{\theta}\|f\|_{L^q(\Rn)}^{1-\theta}.
\label{eq:our_GN_ineq}
\end{equation}
There is equality for functions of the form $f_{GN}\circ A$, there $A\in \GL(n)$ and $f_{GN}$ is of the form \eqref{GN_eq}.
\end{theorem}

\section{The $m$th-order Poincar\'e inequality}
\label{sec:poin}

The methods used to establish the results in this section are inspired by those in \cite{HJM21}. For $\Omega\subset\Rn$, define the spaces $L^p(\Omega), W^{1,p}(\Omega),$ and $W_0^{1,p,q}(\Omega)$ analogously to the spaces $L^p(\Rn), W^{1,p}(\Rn)$ and $W_0^{1,p,q}(\Rn)$, respectively, but with integration over $\Omega$.  Set $W_0^{1,p}(\Omega)=W_0^{1,p,p}(\Omega).$

The main result of this section is the next inequality.  Note that the constant $c_0$ might not be best possible.

\begin{theorem}\label{t:BlSanPoin} Let $\Omega \subset \Rn$ be a bounded domain containing the origin in its interior, $m\in\N,$ $Q\in\conbodio[m]$ be origin-symmetric, and $p \geq 1$. Then, there exists constant $c_0>0$ dependent on $n, p, Q$, and $\Omega$, such that for any $C^1$ not identically zero function $f$ compactly supported in $\Omega$, \begin{equation}\label{e:BlSanPoin}
\mathcal{E}_{p}(Q,f) \geq c_0 \|f\|_{L^p(\Rn)}^{\frac{nm-1}{nm}}\|\nabla f\|_{L^p(\Rn)}^{\frac{1}{nm}}.
\end{equation}
\end{theorem}
\begin{proof}
For simplicity, since the constant $c_0$ in the desired inequality is not best, we will not try to keep track of its precise value, but rather use $c_0$ to denote a ``constant'' that depends on $n, p, Q, \Omega$ which might change from line to line.

Recall the body $\LYZP f$ from \eqref{eq:LYZ_body_2}. Since $Q$ is symmetric, $\LYZP f$ will be too. Arguing as in the proof of Lemma~\ref{lemma 8.15.1} to obtain equations \eqref{eq 8.15.4} and \eqref{eq 8.15.7}, for each $\theta \in \S$, we can find a  $u_{i_*} \in \s$, with $\theta u_{i_*} \neq 0$ and $c_0 >0$ such that 
\[
h_{\LYZ f}(\theta) \geq c_0 \left\|\nabla f(\cdot)^t\frac{\theta u_{i_*}}{|\theta u_{i_*}|}\right\|_{L^p(\Rn)}.
\]
Applying the sharp one-dimensional Poincar\'e inequality given in \cite[page 357]{Talenti1} to the above inequality, we may write 
\begin{equation}\label{eq step 1}
\begin{split}
h_{\LYZ f}(\theta) &\geq c_0 \left(\int_{\Rn}\left| \nabla f(z)^t\frac{\theta u_{i_*}}{|\theta u_{i_*}|}\right|^p dz\right)^{\frac{1}{p}}\\
&= c_0 \left(\int_{\left( \frac{\theta u_{i_*}}{|\theta u_{i_*}|}\right)^{\perp}} \int_{-\infty}^{\infty}\left|\frac{d}{dt} f\left(y+ t\frac{\theta u_{i_*}}{|\theta u_{i_*}|} \right)\right|^p dt dy\right)^{\frac{1}{p}}\\
&\geq c_0 \|f\|_{L^p(\Rn)} w\left(\Omega, \frac{\theta u_{i_*}}{|\theta u_{i_*}|}\right)^{-1}\\
&\geq c_0 \|f\|_{L^p(\Rn)}.
\end{split}
\end{equation}
where $w(\Omega, \xi)$ is the width of $\Omega$ in the direction of $\xi$ and the last line follows from the fact that $\Omega$ is compact.

The inequality \eqref{eq step 1} shows that 
\begin{equation}
    \LYZ f\supset c_0 \|f\|_{L^p(\Rn)}B_2^{nm}.
\end{equation}
On the other hand, observe that 
\begin{equation}\label{eq step 2}
\begin{split}
\max_{\theta \in \mathbb{S}^{nm-1}} h_{\LYZ f}(\theta)^p &= \max_{\theta \in \mathbb{S}^{nm-1}} \|h_Q(\nabla f(\cdot)^t\theta)\|_{L^p(\Rn)}^p\\
&\geq \frac{1}{nm \omega_{nm}} \int_{\mathbb{S}^{nm-1}} \int_{\Rn} h_Q(\nabla f(z)^t\theta)^p dz d\theta\\
&= \frac{1}{nm \omega_{nm}} \int_{\Rn} |\nabla f(z)|^p \int_{\mathbb{S}^{nm-1}} \int_{O(n)} h_Q\left(\left(T^t \frac{\nabla f(z)}{|\nabla f(z)|} \right)^t\theta \right)^p dT d\theta dz\\
&= \frac{1}{n^2m \omega_{nm} \omega_n} \int_{\mathbb{S}^{nm-1}} \int_{\mathbb{S}^{n-1}} h_Q(u^t\theta)^p du d\theta \|\nabla f\|_{L^p(\Rn)}^p\\
&= c_0 \|\nabla f\|_{L^p(\Rn)}^p.
\end{split}
\end{equation}
(Recall here $c_0$ might change from line to line---it is the existence of a positive constant that we are seeking for.)

Equation \eqref{eq step 2} and the fact that $\LYZ f$ is origin-symmetric imply that $\LYZ f$ contains the symmetric segment $\left[-c_0\|\nabla f\|_{L^p(\Rn)}, c_0\|\nabla f\|_{L^p(\Rn)}\right]$. 

Consequently we have that $\LYZ f$ contains the entire double cone 
\[
\text{conv}\left(c_0\|f\|_{L^p(\Rn)} B_2^{nm},\left[-c_0\|\nabla f\|_{L^p(\Rn)}, c_0\|\nabla f\|_{L^p(\Rn)}\right] \right),
\]
and therefore, 
\[
\vol[nm](\LYZ f) \geq c_0 \|f\|_{L^p(\Rn)}^{nm-1}\|\nabla f\|_{L^p(\Rn)}, 
\]
for some $c_0>0$.
Finally, an application of the Blaschke-Santal\'o inequality (see e.g. \cite{MP90}) yields 
\begin{align*}
\mathcal{E}_{p}(Q,f)&=d_{n,p}(Q)\left(\int_{\S} \left( \int_{\Rn} h_Q((\nabla f(z))^t\theta)^pdz \right)^{-\frac{nm} p } d\theta\right)^{-\frac 1 {nm}}\\
& = c_0 \Vol(\LYZP f)^{-\frac{1}{nm}}\geq  c_0 \vol[nm](\LYZ f)^{\frac{1}{nm}}\\
&\geq  c_0\|f\|_{L^p(\Rn)}^{\frac{nm-1}{nm}}\|\nabla f\|_{L^p(\Rn)}^{\frac{1}{nm}}.
\end{align*}
\end{proof}

As an immediate corollary of Theorem~\ref{t:BlSanPoin}, we obtain the $m$th-order Poincar\'e inequality, Theorem~\ref{t:Poincare}, by using an approximation argument and the classical $L^p$ Poincar\'e inequality.

Another corollary of Theorem~\ref{t:BlSanPoin} is the following embedding theorem.  Let $\Omega \subset \Rn$ be a bounded domain containing the origin in its interior, $m\in\N$, $Q\in\conbodio[m]$ be origin-symmetric, and $1 \leq p < \infty$. Consider the class of functions 
\[
\mathcal{B}_{Q,p}(\Omega) := \{f \in W_0^{1,p}(\Omega) \colon \mathcal{E}_{p}(Q,f) \leq 1\}. 
\]

\begin{corollary} For any bounded domain $\Omega \subset \Rn$ with Lipschitz boundary containing the origin in its interior, $m\in\N$, $Q\in\conbodio[m]$ that is origin-symmetric, and $1 \leq p < n$, the set $\mathcal{B}_{Q,p}(\Omega)$ is compactly immersed within $L^p(\Omega)$. 
\end{corollary}

\begin{proof} Consider any sequence of functions $f_k \in W^{1,p}_0(\Omega), k =1,\dots$, such that $\mathcal{E}_p(Q,f_k) \leq 1$. If there is some sub-sequence of $f_{k_j}$ that converges in the $L^p$-norm, then we are done. If this is not the case, then there is a positive constant $c$ for which $\|f_k\|_{L^p(\Omega)}\geq c$ for all $k \geq 1$. However, in this case, according to Theorem~\ref{t:BlSanPoin}, it is necessary that the sequence of $f_k$ is bounded in the $W^{1,p}_0(\Omega)$. Therefore, we may apply the Rellich-Kondarchov embedding theorem to conclude the result. 
\end{proof}

\bibliography{references.bib}
\bibliographystyle{acm}

\end{document}